\documentclass[11pt]{amsart}
\usepackage{amsmath}
\usepackage{amsfonts}
\usepackage{amssymb}
\usepackage{amsthm}
\usepackage{graphicx}
\usepackage{soul}

\usepackage[usenames, dvipsnames]{color} 

\usepackage{hyperref}

\def \de {\partial}
\def \Q {\mathcal{Q}}
\def \N {\mathbb{N}}
\def \O {\Omega}
\def \phi {\varphi}
\def \RNu {\mathbb{R}^{N+1}}
\def \RN {\mathbb{R}^N}
\def \R {\mathbb{R}}
\def \l {\lambda}
\def \L {\Lambda}
\def \e {\epsilon}
\def \Gsb {\Gamma_{s,\beta}}
\def \Cmu {C_0^{-1}(t)}
\def \CmuI {C^{-1}(t)}
\def \la {\left\langle }
\def \ra {\right\rangle}
\def\elle{\mathcal{L}}
\def \trace{\mathrm{tr}}
\def \deter{\mathrm{det}}

\newtheorem{theorem}{Theorem}[section]
\newtheorem{lemma}[theorem]{Lemma}

\newtheorem{corollary}[theorem]{Corollary}
\newtheorem{remark}[theorem]{Remark}
\newtheorem{definition}[theorem]{Definition}

\numberwithin{equation}{section}

\begin{document}

\title[Harnack inequality for a class of Kolmogorov equations]{Harnack inequality for a class of Kolmogorov-Fokker-Planck equations in non-divergence form}
\author[F. Abedin and G. Tralli]{Farhan Abedin and Giulio Tralli$^*$} 
\address{Department of Mathematics, Michigan State University, East Lansing, MI 48823}
\email{abedinf1@msu.edu}
\address{Dipartimento d'Ingegneria Civile e Ambientale (DICEA), Universit\`a di Padova, Via Marzolo 9, 35131 Padova, Italy}
\email{giulio.tralli@unipd.it}
\thanks{
\hphantom{11l}Key words: ultraparabolic equations, apriori estimates, growth lemma.\\
\hphantom{11l}$^*$ Corresponding author: giulio.tralli@unipd.it}

\begin{abstract}
We prove invariant Harnack inequalities for certain classes of non-divergence form equations of Kolmogorov type. The operators we consider exhibit invariance properties with respect to a homogeneous Lie group structure. The coefficient matrix is assumed either to satisfy a Cordes-Landis condition on the eigenvalues, or to admit a uniform modulus of continuity.
\end{abstract}

\maketitle

\setcounter{tocdepth}{1}
\tableofcontents

\section{Introduction}

The purpose of this article is to study regularity properties of solutions to degenerate-parabolic equations in non-divergence form, whose prototypical example is given by
\begin{equation}\label{prototype}
\mathcal{K} := \trace\left(A(v,y,t)D^2_v\right)+\la v, \nabla_y \ra - \de_t, \quad\mbox{for }(v,y,t)\in\R^d\times\R^d\times\R,
\end{equation}
where the $d\times d$ matrix $A(v,y,t)$ is uniformly positive definite. The case $A = \mathbb{I}_d$ corresponds to the well-known Kolmogorov equation \cite{Kol}, which governs the probability density of a system with $2d$ degrees of freedom. The Kolmogorov operator is one of the key examples of hypoelliptic operators studied by H\"ormander in his seminal work \cite{Ho}. Operators like \eqref{prototype} also appear naturally in mathematical finance and various other stochastic models \cite{Ba, Ch, LPP}. Most notably perhaps, the operator $\mathcal{K}$ (and its divergence form counterpart) arises in the kinetic theory of gases as the leading order term in the spatially inhomogeneous Landau equation, which can be interpreted as the limit of the Boltzmann equation when only grazing collisions are taken into account \cite{AV04, Li}. For more on recent progress in the regularity theory of kinetic equations, we refer the reader to the works \cite{CSS, GIMV, GG, IS} and references therein.

In the Landau equation, the coefficients of the principal part depend on the solution itself in a nonlocal manner. This motivates the study of regularity properties of $\mathcal{K}$ with minimal assumptions on the smoothness of $A$. For the divergence-form version of $\mathcal{K}$,
$$\text{div}_v\left(A(v,y,t)\nabla_v\right)+\la v, \nabla_y \ra - \de_t,$$
with bounded measurable coefficients $A$, a Moser-type $L^2$-to-$L^{\infty}$ iteration was obtained in \cite{PP04}, a H\"older regularity result for the solutions was shown in \cite{WZ}, and the Harnack inequality has been proved recently in \cite{GIMV}. Related regularity estimates for a more general class of divergence-form operators with rough coefficients can be found in \cite{CPP, aLP, aLPP}.
For the non-divergence form operator \eqref{prototype} with $A$ assumed to be merely bounded and measurable, the analogue of the Krylov-Safonov Harnack inequality \cite{KS80} is still unknown. This is primarily due to the lack of a suitable version of the Aleksandrov-Bakelman-Pucci maximum principle; we refer to \cite[Chapter VII]{Lieberman} for the uniformly parabolic case. On the other hand, with $A$ assumed to be H\"older continuous, the regularity theory is well-settled and several results have been obtained, even for operators with more general drift terms: we mention, among others, the results concerning the existence of the fundamental solution via Levi-parametrix methods, two-sided Gaussian-type bounds, and also Harnack inequalities \cite{DM, DiP, Poli, PoliA}.

In this work, we prove Harnack's inequality for non-negative solutions to $\mathcal{K}u = 0$ under either a Cordes-Landis condition or a continuity assumption on the coefficient matrix $A$ (see subsection \ref{results}, hypotheses \hyperref[H1]{H1} and \hyperref[H2]{H2}). Similar results have been obtained for other H\"ormander type operators, namely for non-divergence form operators structured on Heisenberg vector fields \cite{AGT, Gutierrez-Tournier-Harnack, Tralli-Critical-Density}.
The techniques we employ in the present work are inspired by the insightful contributions of Landis from the '60s \cite{Landisell}, where he obtained what is nowadays referred to as the \emph{growth lemma} for nonnegative subsolutions of uniformly elliptic equations, assuming that the eigenvalue ratio is close to 1. Glagoleva \cite{glago} established analogous results for uniformly parabolic equations. We refer the reader to the book \cite{Landis} for an exposition of these ideas. In accordance with the literature on ultraparabolic equations, we present our results for operators more general than $\mathcal{K}$ which enjoy invariance properties with respect to a homogeneous Lie group structure. We proceed to describe these operators in more detail and state our main results.

\subsection{Main Results}\label{results}

Fix $N\in\N$. Throughout the paper we denote by $z=(x,t)\in\RN\times\R$ a generic point in $\RNu$. The spatial differential operators will be denoted $\nabla=\nabla_x$,  $D^2=D^2_x$. Fix $p_0, n \in \N$, with $1\leq p_0< N$ and $n\geq 1$. Let $\mathbb{I}_{p_0}$ denote the $p_0 \times p_0$ identity matrix. For some open set $\Omega \subseteq \RNu$, we consider the class of operators
\begin{equation}\label{varintroL}
\elle_A=\trace\left(  A(z) D^2\right)  +\left\langle x, B\nabla\right\rangle - \partial_t \qquad z \in \Omega,
\end{equation}
where $A(z)\in\R^{N\times N}$ is a symmetric nonnegative definite matrix which takes the block form
\begin{equation}\label{A}
A(z)=
\begin{bmatrix}
\mathbb{A}(z) & 0\\
0 & 0
\end{bmatrix}\quad\mbox{ with }\quad \mathbb{A}(z) \in \R^{p_0 \times p_0},
\end{equation}
and
\begin{equation}\label{B}
B=
\begin{bmatrix}
0 & \mathbb{B}_{1} & 0 & \ldots & 0 \\
0  & 0  & \mathbb{B}_{2}  & \ldots  & 0  \\
\vdots & \vdots & \ddots & \ddots & \vdots\\
0 & 0 & 0 & \ldots  & \mathbb{B}_{n} \\
0 & 0 & \ldots & 0  & 0 
\end{bmatrix} \in \R^{N \times N}
\end{equation}
where, for $j = 1, \ldots, n$, $\mathbb{B}_{j}$ is a $p_{j-1}\times p_{j}$ block of rank $p_{j}$, $p_{0}\geq p_{1}\geq \ldots \geq p_{n}\geq1$ and $p_{0}+p_{1}+ \ldots +p_{n}=N$. The matrix $\mathbb{A}(z)$ is assumed to be uniformly positive definite; that is, there exist constants $\lambda, \Lambda > 0$ such that
\begin{equation}\label{ellipticity}
\l\mathbb{I}_{p_0}\leq \mathbb{A}(z) \leq \L\mathbb{I}_{p_0} \qquad\mbox{ for all }z \in \O.
\end{equation}
Notice that the class of operators \eqref{prototype} corresponds to the choices $N=2d, p_0=d, n=1, \mathbb{B}_1=\mathbb{I}_d$.

The conditions on $A(\cdot)$ and $B$ endow the operators $\elle_A$ with rich algebraic properties. As a matter of fact, in the case of a constant matrix $\mathbb{A}$, the operator is of H\"ormander type and the fundamental solution can be written explicitly \cite{Ho, K82}. Moreover, it has been shown in \cite{LP} that this operator is invariant under the action of a homogeneous Lie group, with homogeneous dimension
\begin{equation}\label{defQ}
Q+2:=p_{0}+ 3 p_{1}+...+ (2n +1 ) p_{n} + 2.
\end{equation}
We remark that the presence of homogeneity is tied to the upper triangular form \eqref{B} of the matrix $B$. The group structure allows one to define a homogeneous norm and corresponding cylinder-like sets 
$$\Q^{t_1,t_2}_{r}(z_0), \quad\mbox{for }z_0\in \RNu,\, t_1,t_2\in\R,\, r>0.$$ 
We refer to Section \ref{notcon} for a precise description of all these notions.

To establish Harnack's inequality for the aforementioned operators \eqref{varintroL}, we will assume that the matrix coefficients $A(\cdot)$ satisfy either one of the following hypotheses:
\begin{itemize}
\item[(H1)]\label{H1} {\bf Cordes-Landis assumption:} The coefficients $A(\cdot)$ satisfy the condition \eqref{ellipticity} with
$$\frac{\L}{\l}<1+\frac{2}{Q}.$$
\item[(H2)]\label{H2} {\bf Uniform continuity in $\O$:} The coefficients $A(\cdot)$ admit a uniform modulus of continuity $\omega$ in $\Omega$ (see Definition \ref{modulusofcontinuity}).
\end{itemize}
We can now state our main results. Any constant that depends solely on $B, Q, n, \lambda, \Lambda$ will henceforth be referred to as a \emph{structural} constant.

\begin{theorem}\label{harnackineqH1}(Harnack Inequality under \hyperref[H1]{H1}) Suppose $\elle_A$ satisfies the Cordes-Landis condition \hyperref[H1]{H1}. There exist structural constants $b_B, K, \sigma_0, C > 0$ with $K>\sigma_0$ such that, for all $\Q_{Kr}^{-b_B r^2,0}(z_0)\Subset\Omega$ and $u \in C^2(\Omega)$ satisfying
$$u\geq 0\quad \text{ and } \quad \mathcal{L}_A u = 0\quad\mbox{ in }\,\,\, \Q_{Kr}^{-b_B r^2,0}(z_0),$$
we have
\begin{equation}\label{harnackH1}
\sup_{\Q_r^-} u \leq C \inf_{\Q_r^+} u, 
\end{equation}
where $\Q_r^- := \Q_{\frac{\sigma_0}{2}r}^{-\frac{3b_B}{4}r^2,-\frac{b_B}{2}r^2}(z_0)$ and $\Q_r^+ := \Q_{\frac{\sigma_0}{2}r}^{-\frac{b_B}{4}r^2,0}(z_0)$.
\end{theorem}

\begin{theorem}\label{harnackineqH2}(Harnack Inequality under \hyperref[H2]{H2}) Suppose $\elle_A$ satisfies the uniform continuity assumption \hyperref[H2]{H2} in $\Omega$, with modulus of continuity $\omega$. There exist positive constants $b_B, K, \sigma, C > 0$ depending on $\omega$ and on structural constants such that, for all $\Q_{Kr}^{-b_B r^2,0}(z_0)\Subset\Omega$ with $0<r\leq 1$ and $u \in C^2(\Omega)$ satisfying
$$u\geq 0\quad \text{ and } \quad \mathcal{L}_A u = 0\quad\mbox{ in }\,\,\, \Q_{Kr}^{-b_B r^2,0}(z_0),$$
we have
\begin{equation}\label{harnackH2}
\sup_{\Q_r^-} u \leq C \inf_{\Q_r^+} u, 
\end{equation}
where $\Q_r^- := \Q_{\sigma r}^{-\frac{3b_B}{4}r^2,-\frac{5b_B}{8}r^2}(z_0)$ and $\Q_r^+ := \Q_{\sigma r}^{-\frac{b_B}{8}r^2,0}(z_0)$.
\end{theorem}

We point out that Theorem \ref{harnackineqH1} is, to the best of our knowledge, the first regularity result for non-divergence form operators like $\mathcal{K}$ that is independent of the smoothness of the coefficients. Theorem \ref{harnackineqH2} also generalizes, in the case of the homogeneous operators \eqref{varintroL}, the Harnack inequality obtained in \cite{Poli} (see also \cite{DiP}) assuming H\"older continuity of the coefficients. 

The essential ingredients in the proof of Theorems \ref{harnackineqH1} and \ref{harnackineqH2} are, respectively, Theorem \ref{growthlemma} and Theorem \ref{growthlemmaH2}. These are the analogues of the classical growth lemma of Landis, and they establish pointwise-to-measure estimates for nonnegative subsolutions to $\elle_A$ in a quantitative manner. In order to establish these key estimates, we construct barriers using the potentials generated by kernels resembling the fundamental solution for constant coefficient operators. This involves a careful estimate of the aforementioned kernels in terms of the length scale of the cylinders. It is only in the construction of these barriers where we use the hypotheses \hyperref[H1]{H1} and \hyperref[H2]{H2}. Once the required pointwise-to-measure estimates are established, there are, by now, standard ways in the literature to proceed with the proof of Harnack's inequality. In this work, we have chosen to follow the general approach outlined by Landis in \cite{Landis}. For this strategy to succeed, we must deal with the non-standard nature of the cylinder-like sets $\Q^{t_1,t_2}_{r}(z_0)$.

The outline of this paper is as follows. In Section \ref{notcon}, we set up our notation and recall some properties of the relevant geometric objects. In Section \ref{secest}, we establish upper and lower bounds for the kernels \eqref{defbar}. We then use these kernels in Section \ref{barriers} to construct barriers for $\elle_A$ under the hypotheses \hyperref[H1]{H1} and \hyperref[H2]{H2} (see respectively subsections \ref{cordeslandiscondition} and \ref{continuityassumption}). In Section \ref{sec5}, we prove the growth lemmas (Theorems \ref{growthlemma} and \ref{growthlemmaH2}), and provide as application the oscillation decay and the H\"older continuity of solutions to $\elle_A u = 0$. Finally, in Section \ref{harn} we complete the proofs of Theorems \ref{harnackineqH1} and \ref{harnackineqH2}.

\section{Preliminaries}\label{notcon}

The block structure \eqref{B} on the matrix $B$ implies (see \cite[Section 2]{LP}) the following H\"ormander rank condition \cite{Ho}:
$${\rm{rank\,\, Lie\, }}\left\{\de_{x_1},\ldots,\de_{x_{p_0}},\left\langle x, B\nabla\right\rangle-\de_t\right\}(z)=N+1\quad\forall\,z\in\RNu.$$
In particular, for any constant matrix $A_0\in\R^{N\times N}$ with the block structure
\begin{equation}\label{azero}
A_0=
\begin{bmatrix}
\mathbb{A}_{0} & 0\\
0 & 0
\end{bmatrix}\quad\mbox{ with }\quad \l\mathbb{I}_{p_0}\leq \mathbb{A}_{0}\leq \L\mathbb{I}_{p_0},
\end{equation}
the operator 
$$\elle_0=\trace\left(  A_0 D^2\right)  +\left\langle x, B\nabla\right\rangle - \partial_t$$
is hypoelliptic. The stationary part of the operator $\elle_0$ is the infinitesimal generator of a Gaussian process with covariance matrix given by
$$C_0(t)=\int_{0}^{t} E(\sigma)A_0E^T(\sigma)\,d\sigma,$$
where
\begin{equation}\label{matrixexponential}
E(\sigma) :=\exp\left(-\sigma B^T\right), \qquad \sigma\in\R.
\end{equation}
The H\"ormander rank condition is actually equivalent (see \cite{Ho, LP}) to the following Kalman-type condition
$$C_0(t)>0\qquad\forall\,t>0.$$
Throughout the paper we will use the notation
\begin{equation}\label{Kalma0}
I_0=\begin{bmatrix}
\mathbb{I}_{p_0} & 0\\
0 & 0
\end{bmatrix}\in \R^{N \times N} \quad\mbox{ and }\quad C(t)=\int_{0}^{t} E(\sigma)I_0E^T(\sigma)\,d\sigma .
\end{equation}
The assumption \eqref{ellipticity} for the coefficient matrix $A(z)$ of the operators $\elle_A$ in \eqref{varintroL}-\eqref{A} is clearly equivalent to assuming
\begin{equation}\label{ellipticityofA}
\l I_0\leq A(z) \leq \L I_0 \qquad \text{ for all }z \in \O.
\end{equation}

Let us now describe the group structure mentioned in the Introduction. We refer the reader to \cite[Section 1]{LP} for a complete exposition. Recalling \eqref{matrixexponential}, the group law is given by

$$z  \circ\zeta  =\left(  \xi+E(\tau)  x,t+ \tau\right),\quad \mbox{for }z=\left(  x,t\right), \zeta=\left(  \xi,\tau\right)\in\RNu.$$
Moreover, recalling the $p_i$'s coming from the structure of $B$ in \eqref{B}, we can denote any $x \in \RN$ as
$$x=\left(x^{(p_0)},x^{(p_1)},\ldots,x^{(p_n)}\right)\in \R^{p_0}\times\R^{p_1}\times\cdots\times\R^{p_n} = \RN,$$
and we can define the family of group automorphisms $\left(\delta_r\right)_{r>0}$ as
\begin{eqnarray*}
\delta_r&:&\RNu\longrightarrow\RNu\\ 
\delta_r(x,t)&=&\left(r
x^{(p_0)},r^3
x^{(p_1)},\ldots,r^{2n+1}x^{(p_n)},r^2t\right).
\end{eqnarray*}
These will play the role of homogeneous dilations. For convenience, we also denote the spatial dilations by
$$D_r:\RN\longrightarrow\RN,\quad D_r(x) = \left(r
x^{(p_0)},r^3
x^{(p_1)},\ldots,r^{2n+1}x^{(p_n)}\right).$$
The fact that $\delta_r$ are automorphisms with respect to $\circ$ is encoded in the following commutation property (see \cite[equation (2.20)]{LP} and \cite{K82})
\begin{equation}\label{commu}
E(r^2\sigma)=D_rE(\sigma)D_{\frac{1}{r}}\quad\mbox{ for any }r>0\,\mbox{ and }\,\sigma\in\R.
\end{equation}
From this, one can deduce that the covariance matrix $C_0(t)$ satisfies the commutation relation
\begin{equation}\label{split}
C_0(t)= D_{\sqrt{t}} C_0(1) D_{\sqrt{t}}. 
\end{equation}
If $Q$ is the number defined in \eqref{defQ} and $|\cdot|$ denotes Lebesgue measure (both in $\RNu$ and $\RN$), then we have
\begin{equation}\label{invmeasure}
|\delta_r(E)|=r^{Q+2}|E|,\quad |D_r(F)|=r^{Q}|F|,\quad |z_0\circ E|=|E|
\end{equation}
for all $z_0\in\RNu$, $r>0$, and for any Lebesgue measurable sets $E\subset\RNu$, $F\subset\RN$. In \cite{LP} it is shown that the vector fields $\left\{\de_{x_1},\ldots,\de_{x_{p_0}},\left\langle x, B\nabla\right\rangle-\de_t\right\}$ are left-translation invariant and $\delta_r$-homogeneous (respectively of degree $1$ and $2$). Consequently, the operators $\elle_0$ are left-translation invariant and $\delta_r$-homogenous of degree $2$. 

One can associate to this homogeneous structure a family of cylinder-like sets. Denoting also the Euclidean norms in $\R^N$, $\R^{p_k}$ or $\R$ by $\left|\cdot \right|$, we can define the norms $\left|\cdot \right|_B:\RN\longrightarrow\R^+$ and  $\left\|\cdot\right\|_B:\R^{N+1}\longrightarrow\R^+$ by
$$|x|_B=\sum_{i=0}^n\left|x^{(p_i)}\right|^{\frac{1}{2i+1}},\qquad\mbox{for }x=\left(x^{(p_0)},\ldots,x^{(p_n)}\right)\in\R^{p_0}\times\cdots\times\R^{p_n}=\RN,$$
$$\left\|z\right\|_B=|x|_B + |t|^{1/2}, \qquad\mbox{for } z = (x,t) \in \mathbb{R}^{N+1}.$$
The subscript $B$ is used to distinguish the homogeneous norm $\left\|\cdot \right\|_B$ from the matrix norm $\left\|\cdot \right\|$.
Note that $\left|\cdot \right|_B$ and $\left\|\cdot\right\|_B$ are respectively $D_r$-homogeneous and $\delta_r$-homogeneous functions of degree $1$.
The homogeneous ball of radius $r>0$ centered at $0$ is the set
$$B_{r}(0): =\left\{x\in\RN\,:\,|x|_B <  r\right\}=D_{r}\left(B_1(0)\right).$$
The cylinder-like sets centered at $0$ are defined as
$$\Q^{t_1, t_2}_{r} = B_{r}(0) \times (t_1, t_2)$$
where $r>0$ and $t_1<t_2\in\R$. Cylinder-like sets centered at an arbitrary point $z_0 \in \RNu$ are defined as
$$\Q^{t_1, t_2}_{r}(z_0): =z_0\circ \Q^{t_1, t_2}_{r}.$$
It is clear from \eqref{invmeasure} and the composition and dilation laws that, for any $b>0$,
$$\left|\Q^{t_1, t_1+br^2}_{r}(z_0)\right|=r^{Q+2}\left|\Q^{0, b}_{1}\right|\quad\mbox{ for all }z_0\in\RNu, t_1\in\R, r>0.$$
The notion of parabolic boundary of a cylinder can be naturally extended to this setting, and is defined as
$$\de_p\Q^{t_1, t_2}_{r}:=\left(B_r(0)\times\{t_1\}\right)\cup \left(\de B_r(0)\times [t_1,t_2] \right)\quad\mbox{ and }\quad\de_p\Q^{t_1, t_2}_{r}(z_0):=z_0\circ \de_p\Q^{t_1, t_2}_{r}.$$
It is easy to check that $\de_p\Q^{t_1, t_2}_{r}(z_0)=\overline{\de\Q^{t_1, t_2}_{r}(z_0)\cap\{t<t_2+t_0\}}$.
We can now state the analogue of the parabolic weak maximum principle for the operators $\elle_A$ in \eqref{varintroL}, whose proof is, by now, classical for degenerate-parabolic equations. 
\begin{center} Let $T\in\R$ and let $D\subset\RNu$ be a bounded open set; if $v\in C^2(D)\cap C\left(\overline{D}\right)$ satisfies
\begin{equation}\label{MP}
\begin{cases}
\elle_A v\geq 0 & \text{ in }\,D\cap\{t<T\},\\
v\leq 0 & \text{ on }\,\de D\cap\{t<T\},
\end{cases}
\qquad\mbox{ then }\qquad v\leq 0 \,\,\text{ in }\,\, D\cap\{t<T\}.
\end{equation}
\end{center}
We recall a number of essential relations between the homogeneous norm $|\cdot|_B$ and the Euclidean norm that will be used throughout the paper. Some of these can already be found in \cite{KLT, Manfr, Poli}; we collect and prove them in the following lemma for the reader's convenience.
\begin{lemma}\label{normsarecomparable} The following properties hold:
\begin{itemize}
\item[(i)] The triangle inequality holds in the norm $|\cdot|_B$:
\begin{equation}\label{triangleB}
|x+\xi|_B\leq |x|_B + |\xi|_B\qquad \forall\,x,\xi\in\RN.
\end{equation}
\item[(ii)] Denoting $\sigma_0=\min_{\left|x\right|=1}{|x|_B}$ and $\bar{\sigma}=\max_{\left|x\right|=1}{|x|_B}$ we have
\begin{equation}\label{confrhomnonhom}
\sigma_0\min{\left\{\left|x\right|, \left|x\right|^{\frac{1}{2n+1}}\right\}}\leq|x|_B\leq \bar{\sigma}\max{\left\{\left|x\right|, \left|x\right|^{\frac{1}{2n+1}}\right\}}\qquad \forall\,x\in\RN.
\end{equation}
\item[(iii)] There exists a structural constant $c(n,B)>0$ such that
\begin{equation}\label{benne}
\left|\left(E(t)-\mathbb{I}_N\right) x\right|_B\leq  c(n,B)\max\left\{|x|^{\frac{1}{3}},|x|^{\frac{1}{2n+1}}\right\}\max\left\{|t|^{\frac{1}{2n+1}},|t|^{\frac{n}{2n+1}}\right\} \qquad \forall\,x\in\RN,\,\,\forall\, t\in\R.
\end{equation}
\end{itemize}
\end{lemma}

\begin{proof} 
\begin{itemize}
\item[(i)] This follows from the subadditivity of $|\cdot|^p$ for $0 < p < 1$.
\item[(ii)] For $x=0$ this is trivial. For any $x\neq 0$, we have
$$\frac{|x|_B}{\max{\left\{\left|x\right|,\left|x\right|^{\frac{1}{2n+1}}\right\}}}\leq\sum_{i=0}^n{\frac{\left|x^{(p_i)}\right|^{\frac{1}{2i+1}}}{\left|x\right|^{\frac{1}{2i+1}}}}=\sum_{i=0}^n{\left|\left(\frac{x}{\left|x\right|}\right)^{(p_i)}\right|^{\frac{1}{2i+1}}}=\left|\frac{x}{\left|x\right|}\right|_B \leq\bar{\sigma}.$$
while on the other side,
$$\frac{\left|x\right|_B}{\min{\left\{\left|x\right|,\left|x\right|^{\frac{1}{2n+1}}\right\}}}\geq\sum_{i=0}^n{\frac{\left|x^{(p_i)}\right|^{\frac{1}{2i+1}}}{\left|x\right|^{\frac{1}{2i+1}}}}=\sum_{i=0}^n{\left|\left(\frac{x}{\left|x\right|}\right)^{(p_i)}\right|^{\frac{1}{2i+1}}}=\left|\frac{x}{\left|x\right|}\right|_B \geq\sigma_0.$$
\item[(iii)]  Fix any $x\in\RN$, $t\in\R$. By the upper triangular form of $B$, we have $\left(E(t)x\right)^{(p_0)}=x^{(p_0)}$ and for any $i\in\{1,\ldots,n\}$
$$\left(E(t)x\right)^{(p_i)}=x^{(p_i)}+\sum_{k=1}^i\frac{(-t)^k}{k!}\left(\mathbb{B}^T_i\mathbb{B}^T_{i-1}\cdots\mathbb{B}^T_{i-k+1}\right)x^{(p_{i-k})}.$$
Hence, by denoting $M_B=\max_{i}\left\|\mathbb{B}^T_i\right\|$, we get
\begin{eqnarray*}
\left|\left(E(t)-\mathbb{I}_N\right) x\right|_B&=&\sum_{i=1}^n\left|\sum_{k=1}^i\frac{(-t)^k}{k!}\left(\mathbb{B}^T_i\mathbb{B}^T_{i-1}\cdots\mathbb{B}^T_{i-k+1}\right)x^{(p_{i-k})}\right|^{\frac{1}{2i+1}}\\
&\leq& \sum_{i=1}^n\sum_{k=1}^i|t|^{\frac{k}{2i+1}}M_B^{\frac{k}{2i+1}}\left|x^{(p_{i-k})}\right|^{\frac{1}{2i+1}}\\
&\leq&\frac{n(n+1)}{2}\max\left\{M_B^{\frac{1}{2n+1}},M_B^{\frac{n}{2n+1}}\right\}\max\left\{|x|^{\frac{1}{3}},|x|^{\frac{1}{2n+1}}\right\}\max\left\{|t|^{\frac{1}{2n+1}},|t|^{\frac{n}{2n+1}}\right\}.
\end{eqnarray*}
\end{itemize}
\end{proof}
It is known \cite{Ho, K82} that the fundamental solution of $\elle_0$ with pole at the origin is given by 
\begin{equation}\label{fondfrozen}
\Gamma_0 \left(x,t\right)  =\left\{
\begin{tabular}
[c]{lll}
$0$ & $\text{for }t\leq0,$\\ \mbox{}\\
$\frac{(4\pi)^{-\frac{N}{2}}}{\sqrt{\deter(C_0(t))}} \exp \left(  -\frac{1}{4}  \langle C_0^{-1} (t) x,   x \rangle \right)=\frac{c_0}{t^{\frac{Q}{2}}} \exp \left(  -\frac{1}{4}  \langle C_0^{-1} (1) D_{\frac{1}{\sqrt{t}}}  x,  D_{\frac{1}{\sqrt{t}}}  x \rangle \right)  $ & $\text{for } t>0,$
\end{tabular}
\ \right. 
\end{equation}
where $c_0=(4\pi)^{-\frac{N}{2}}\left(\deter(C_0(1))\right)^{-\frac{1}{2}}$. By the translation invariance of $\elle_0$, one can relocate the pole to any desired point. Note also that $\Gamma_0$ is $\delta_r$-homogeneous of degree $-Q$. The fundamental solution $\Gamma_0$ and its level sets play an essential role in the proof of Harnack's inequality for the operator $\elle_0$ established in \cite{GL90, LP}. In the sequel, it will be necessary for us to have good estimates on the quadratic form $\langle C_0^{-1} (t) x, x \rangle$. We begin by defining a relevant structural constant $b_B$. Since
$E(\sigma)=e^{-\sigma B^T}\rightarrow\mathbb{I}_N$ as $\sigma\rightarrow 0$,
we can define the constant $b_B$ such that
\begin{equation}\label{bibi}
0<b_B\leq\left(\frac{\sigma_0}{\bar{\sigma}}\right)^2\qquad\mbox{and}\qquad \left\|E(\sigma)\right\|\leq 2\,\,\,\,\mbox{ for all }|\sigma|\leq b_B.
\end{equation}
Here the constants $\sigma_0\leq\bar{\sigma}$ are the ones from \eqref{confrhomnonhom}.
\begin{lemma}\label{evalsofC}
There exist structural constants $\L_1, \l_1$ such that
\begin{equation}\label{bici}
\frac{1}{\L_1 t}\mathbb{I_N}\leq \Cmu \leq \frac{1}{\l_1 t^{2n+1}} \mathbb{I_N}\quad\mbox{ for all }0<t\leq b_B.
\end{equation}
\end{lemma}
\begin{proof} Fix an arbitrary $v\in\RN$ with $|v|=1$. If $0<t\leq b_B$, then it follows from \eqref{confrhomnonhom} and \eqref{bibi}
$$\min{\left\{\left|D_{\sqrt{t}} v\right|, \left|D_{\sqrt{t}} v\right|^{\frac{1}{2n+1}}\right\}}\leq\frac{1}{\sigma_0}\left|D_{\sqrt{t}} v\right|_B=\frac{\sqrt{t}}{\sigma_0}\left|v\right|_B\leq \sqrt{t}\frac{\bar{\sigma}}{\sigma_0}\leq 1.$$
This says in particular that $\left|D_{\sqrt{t}} v\right|\leq \left|D_{\sqrt{t}} v\right|^{\frac{1}{2n+1}}$, and so
\begin{equation}\label{upv}
\left|D_{\sqrt{t}} v\right|\leq \sqrt{t}\frac{\bar{\sigma}}{\sigma_0}.
\end{equation}
On the other side, we can use \eqref{confrhomnonhom} again to obtain
\begin{equation}\label{dowv}
\left|D_{\sqrt{t}} v\right|^{\frac{1}{2n+1}}=\max{\left\{\left|D_{\sqrt{t}} v\right|, \left|D_{\sqrt{t}} v\right|^{\frac{1}{2n+1}}\right\}}\geq\frac{1}{\bar{\sigma}}\left|D_{\sqrt{t}} v\right|_B=\frac{\sqrt{t}}{\bar{\sigma}}\left|v\right|_B \geq \sqrt{t}\frac{\sigma_0}{\bar{\sigma}}.
\end{equation}
We can now employ the commutation relation \eqref{split} and the hypothesis \eqref{azero} on $A_0$ to uniformly bound from above and below the quadratic form $\la C_0(t)v,v\ra$. Denote by $\L_I$ and $\l_I$ respectively the maximum and the minimum eigenvalue of $C(1)$. Then by \eqref{upv} and \eqref{dowv}, we get
\begin{eqnarray*}
&&\left\langle C_0(t) v,v \right\rangle\leq \L \left\langle C(1) D_{\sqrt{t}} v,D_{\sqrt{t}} v \right\rangle \leq\L
\L_I|D_{\sqrt{t}} v|^2\leq \L\L_I\left(\frac{\bar{\sigma}}{\sigma_0}\right)^2 t=:\L_1 t,\\
&&\left\langle C_0(t) v,v \right\rangle \geq \l \left\langle C(1) D_{\sqrt{t}} v,D_{\sqrt{t}} v \right\rangle
\geq \l\l_I|D_{\sqrt{t}} v|^2\geq \l\l_I\left(\frac{\sigma_0}{\bar{\sigma}}\right)^{4n+2} t^{2n+1}=:\l_1 t^{2n+1},
\end{eqnarray*}
for every $v\in\RN$ with $|v|=1$. In other words, we have just shown that
$$\l_1 t^{2n+1}\mathbb{I_N}\leq C_0(t) \leq \L_1 t \mathbb{I_N} \quad\mbox{ for all }0<t\leq b_B$$
for some structural constants $\l_1,\L_1$. This implies
\begin{equation*}
\frac{1}{\L_1 t}\mathbb{I_N}\leq \Cmu \leq \frac{1}{\l_1 t^{2n+1}} \mathbb{I_N}\quad\mbox{ for all }0<t\leq b_B.
\end{equation*}\end{proof}

\section{Pointwise estimates for Gaussian Kernels}\label{secest}

In this section, we initiate the construction of explicit barriers which will be used to prove the growth lemma. These barriers are modeled after the fundamental solution $\Gamma_0$ \eqref{fondfrozen}. To this end, for $s, \beta > 0$, we consider the function
\begin{equation}\label{defbar}
\Gsb(z ) =\left\{
\begin{tabular}
[c]{lll}
$0$ & $\text{for }t\leq0,$\\\mbox{}\\
$\frac{1}{t^{s\frac{Q}{2}}  }
\exp\left(  -\frac{1}{4\beta}\left\langle C_0^{-1}(  1)D_{\frac{1}{\sqrt{t}}}  x,D_{\frac{1}{\sqrt{t}}} x\right\rangle
\right)  $ & $\text{for }t>0.$
\end{tabular}
\ \right.
\end{equation}
Note that $\Gsb$ is $\delta_r$-homogeneous of degree $-sQ$. We devote the rest of this section to establishing the necessary pointwise estimates for $\Gsb$.

\begin{lemma}\label{upperboundlemma} (Upper Bound for $\Gsb$)
Let $s,\beta$ be positive numbers. There exist $c_1>0$ and $K_1> \sigma_0$ depending just on $s, \beta$, and structural constants such that, for every $r>0$ and $K\geq K_1$, if we consider the cylindrical sets
$$\Q^1_r:=\Q_{K r}^{-b_B r^2, 0},
\qquad S^1_r:=\de B_{Kr}(0)\times [-b_B r^2,0],\qquad \Q^3_r:=\Q_{\sigma_0 r}^{-b_B r^2, -\frac{1}{2}b_B r^2},
$$
then we have
\begin{equation}\label{upperbound}
\Gsb(\zeta^{-1}\circ z)\leq \frac{1}{(b_B r^2)^{s\frac{Q}{2}}  }
\exp\left(\frac{-c_1 K^2}{b_B} \right)\quad\forall\,\zeta=(\xi,\tau)\in \Q^3_r, \ z=(x,t)\in S^1_r.
\end{equation}
\end{lemma}
\begin{proof} Recall that $\Gsb(\zeta^{-1}\circ z)=0$ if $t\leq \tau$. Therefore, it suffices to assume $z\in S^1_r$ and $\zeta\in \Q^3_r$ with $t>\tau$. In this case, we have
$$\Gsb(\zeta^{-1}\circ z)=\frac{1}{(t-\tau)^{s\frac{Q}{2}}  }
\exp\left(  -\frac{1}{4\beta}\left\langle C_0^{-1}(t-\tau)\left(x-E(t-\tau)\xi\right),\left(x-E(t-\tau)\xi\right)\right\rangle\right).$$
Let us deal with the term inside the exponential. By \eqref{split}, we have
\begin{eqnarray*}
&&\left\langle C_0^{-1}(t-\tau)\left(x-E(t-\tau)\xi\right),\left(x-E(t-\tau)\xi\right)\right\rangle\\
&&=\left\langle C_0^{-1}\left(\frac{t-\tau}{r^2}\right)D_{\frac{1}{r}} \left(x-E(t-\tau)\xi\right),D_{\frac{1}{r}} \left(x-E(t-\tau)\xi\right)\right\rangle \qquad \mbox{ for all } r>0.
\end{eqnarray*}
By definition we have $0<\frac{t-\tau}{r^2}\leq \frac{-\tau}{r^2}\leq b_B$. Therefore,
\begin{eqnarray*}
&&\left\langle C_0^{-1}\left(\frac{t-\tau}{r^2}\right)D_{\frac{1}{r}} \left(x-E(t-\tau)\xi\right),D_{\frac{1}{r}} \left(x-E(t-\tau)\xi\right)\right\rangle \\
&\geq& \frac{r^2}{(t-\tau)\L_1}|D_{\frac{1}{r}} \left(x-E(t-\tau)\xi\right)|^2 \qquad \mbox{ by } \eqref{bici} \\
&=&\frac{r^2}{(t-\tau)\L_1}\left(|D_{\frac{1}{r}} x|^2-2\left\langle D_{\frac{1}{r}} x, D_{\frac{1}{r}} E(t-\tau)\xi \right\rangle + |D_{\frac{1}{r}} E(t-\tau)\xi|^2\right)\\
&\geq& \frac{r^2}{(t-\tau)\L_1}\left(|D_{\frac{1}{r}} x|^2-2\left\langle D_{\frac{1}{r}} x, D_{\frac{1}{r}} E(t-\tau)\xi \right\rangle\right)\\
&=& \frac{r^2}{(t-\tau)\L_1}\left(|D_{\frac{1}{r}} x|^2-2\left\langle D_{\frac{1}{r}} x, E\left(\frac{t-\tau}{r^2}\right)D_{\frac{1}{r}}\xi \right\rangle\right) \qquad \mbox{ by } \eqref{commu} \\
&\geq&\frac{r^2}{(t-\tau)\L_1}\left(|D_{\frac{1}{r}} x|^2-2\left\|E\left(\frac{t-\tau}{r^2}\right)\right\| |D_{\frac{1}{r}} x| |D_{\frac{1}{r}}\xi|\right)\\
&\geq& \frac{r^2}{(t-\tau)\L_1}|D_{\frac{1}{r}} x|\left(|D_{\frac{1}{r}} x|-4 |D_{\frac{1}{r}}\xi|\right) \qquad \mbox{ by } \eqref{bibi}.
\end{eqnarray*}
In summary, we have just proved
\begin{equation}\label{belowexp}
\left\langle C_0^{-1}(t-\tau)\left(x-E(t-\tau)\xi\right),\left(x-E(t-\tau)\xi\right)\right\rangle\geq \frac{r^2}{(t-\tau)\L_1}|D_{\frac{1}{r}} x|\left(|D_{\frac{1}{r}} x|-4 |D_{\frac{1}{r}}\xi|\right).
\end{equation}
We now need a bound from below for $|D_{\frac{1}{r}} x|$ and a bound from above for $|D_{\frac{1}{r}} \xi|$. By \eqref{confrhomnonhom} and the definition of $S^1_r$, we have
$$K_1\leq K=\left|D_{\frac{1}{r}} x\right|_B\leq \bar{\sigma}\max{\left\{\left|D_{\frac{1}{r}}x\right|, \left|D_{\frac{1}{r}}x\right|^{\frac{1}{2n+1}}\right\}}.$$
If we choose $K_1\geq \bar{\sigma}$, this yields 
\begin{equation}\label{drx}
|D_{\frac{1}{r}} x|=\max{\left\{\left|D_{\frac{1}{r}}x\right|, \left|D_{\frac{1}{r}}x\right|^{\frac{1}{2n+1}}\right\}}\geq \frac{K}{\bar{\sigma}}.
\end{equation}
On the other hand, by \eqref{confrhomnonhom} and the definition of $\Q^3_r$, we have
$$\sigma_0\geq \left|D_{\frac{1}{r}} \xi\right|_B\geq \sigma_0\min{\left\{\left|D_{\frac{1}{r}} \xi\right|, \left|D_{\frac{1}{r}} \xi\right|^{\frac{1}{2n+1}}\right\}},$$
which says in particular that
\begin{equation}\label{drxi}
|D_{\frac{1}{r}} \xi|=\min{\left\{\left|D_{\frac{1}{r}} \xi\right|, \left|D_{\frac{1}{r}} \xi\right|^{\frac{1}{2n+1}}\right\}}\leq 1.
\end{equation}
Using \eqref{drx} and \eqref{drxi} in \eqref{belowexp}, and choosing $K_1\geq 8\bar{\sigma}$, we get
$$\left\langle C_0^{-1}(t-\tau)\left(x-E(t-\tau)\xi\right),\left(x-E(t-\tau)\xi\right)\right\rangle\geq \frac{r^2}{t-\tau}\frac{K^2}{2\L_1(\bar{\sigma})^2}.$$
Hence
$$\Gsb(\zeta^{-1}\circ z)\leq\frac{1}{(t-\tau)^{s\frac{Q}{2}}  }
\exp\left( \frac{-r^2}{t-\tau}\frac{K^2}{8\L_1(\bar{\sigma})^2\beta} \right).$$
The function $f(y)=y^{-s\frac{Q}{2}}\exp{\left(\frac{-C}{y}\right)}$ (for $C>0$) is monotone increasing for $y\in(0,\frac{2C}{sQ}]$. If we finally fix
$$K_1=\max\left\{8\bar{\sigma}, 2\bar{\sigma}\sqrt{b_B\L_1\beta s Q}\right\},$$
then $0<t-\tau\leq b_Br^2\leq \frac{r^2K^2}{4\L_1(\bar{\sigma})^2\beta s Q }$ by construction and we get
$$\Gsb(\zeta^{-1}\circ z)\leq\frac{1}{(b_B r^2)^{s\frac{Q}{2}}  }
\exp\left(\frac{-K^2}{8\L_1(\bar{\sigma})^2\beta b_B} \right).$$
\end{proof}

\begin{lemma}\label{lowerboundlemma} (Lower Bound for $\Gsb$)
Let $s,\beta, r$ be positive numbers. Consider the cylinders
$$\Q^2_r:=\Q_{\sigma_0 r}^{-\frac{1}{4}b_B r^2, 0},\qquad 
\Q^3_r:=\Q_{\sigma_0 r}^{-b_B r^2, -\frac{1}{2}b_B r^2}.$$
There exists $c_2>0$ depending on $\beta$ and structural constants such that
\begin{equation}\label{lowerbound}
\Gsb(\zeta^{-1}\circ z)\geq \frac{1}{(b_B r^2)^{s\frac{Q}{2}}  }
\exp\left(\frac{-c_2}{b_B^{2n+1}}\right)\quad\forall\,\zeta=(\xi,\tau)\in \Q^3_r,\,\ z=(x,t)\in \Q^2_r.
\end{equation}
\end{lemma}
\begin{proof}
Fix $z\in \Q^2_r$ and $\zeta\in \Q^3_r$. Then
\begin{equation}\label{tmentau}
b_Br^2\geq t-\tau\geq \frac{1}{4}b_Br^2.
\end{equation}
We argue similarly to the proof of Lemma \ref{upperboundlemma} to estimate the quadratic form from below:
\begin{eqnarray*}
&&\left\langle C_0^{-1}(t-\tau)\left(x-E(t-\tau)\xi\right),\left(x-E(t-\tau)\xi\right)\right\rangle=\\
&=&\left\langle C_0^{-1}\left(\frac{t-\tau}{r^2}\right)D_{\frac{1}{r}} \left(x-E(t-\tau)\xi\right),D_{\frac{1}{r}} \left(x-E(t-\tau)\xi\right)\right\rangle \qquad \mbox{ by } \eqref{split}\\
&\leq&\frac{1}{\l_1}\left(\frac{r^2}{t-\tau}\right)^{2n+1}|D_{\frac{1}{r}} \left(x-E(t-\tau)\xi\right)|^2 \qquad \mbox{ by } \eqref{bici}\\
&\leq& \frac{1}{\l_1}\left(\frac{4}{b_B}\right)^{2n+1}|D_{\frac{1}{r}} \left(x-E(t-\tau)\xi\right)|^2 \qquad \mbox{ by } \eqref{tmentau}\\
&\leq& \frac{2}{\l_1}\left(\frac{4}{b_B}\right)^{2n+1}\left(|D_{\frac{1}{r}}x|^2 + |D_{\frac{1}{r}}E(t-\tau)\xi|^2\right)\\
&=& \frac{2}{\l_1}\left(\frac{4}{b_B}\right)^{2n+1}\left(|D_{\frac{1}{r}}x|^2 + \left|E\left(\frac{t-\tau}{r^2}\right)D_{\frac{1}{r}}\xi\right|^2\right) \qquad \mbox{ by } \eqref{commu}\\
&\leq&\frac{2}{\l_1}\left(\frac{4}{b_B}\right)^{2n+1}\left(|D_{\frac{1}{r}}x|^2 + \left\|E\left(\frac{t-\tau}{r^2}\right)\right\|^2|D_{\frac{1}{r}}\xi|^2\right)\\
&\leq&\frac{2}{\l_1}\left(\frac{4}{b_B}\right)^{2n+1}\left(|D_{\frac{1}{r}}x|^2 + 4|D_{\frac{1}{r}}\xi|^2\right) \qquad \mbox{ by } \eqref{tmentau}\mbox{ and }\eqref{bibi}.
\end{eqnarray*}
Since $x,\xi\in B_{\sigma_0 r}(0)$, we can argue as in \eqref{drxi} to conclude
$$|D_{\frac{1}{r}}x|, \ |D_{\frac{1}{r}}\xi|\leq 1.$$
Hence
$$\left\langle C_0^{-1}(t-\tau)\left(x-E(t-\tau)\xi\right),\left(x-E(t-\tau)\xi\right)\right\rangle\leq\frac{10}{\l_1}\left(\frac{4}{b_B}\right)^{2n+1}.$$
Recalling once more \eqref{tmentau}, we finally obtain
$$\Gsb(\zeta^{-1}\circ z)\geq \frac{1}{(b_B r^2)^{s\frac{Q}{2}}  }
\exp\left(\frac{-5}{2\l_1\beta} \left(\frac{4}{b_B}\right)^{2n+1} \right).$$
\end{proof}

\section{Construction of Barriers}\label{barriers}

Our aim in this section is to construct potentials using the kernels $\Gsb$ \eqref{defbar}. The eventual goal is to use these potentials as barriers for the operators $\elle_A$ under the the hypotheses \hyperref[H1]{H1} and \hyperref[H2]{H2}. The pointwise bounds from Section \ref{secest} will then allow us to successfully use comparison principle arguments in the proof of the growth lemma given in Section \ref{sec5}.

For a fixed Lebesgue-measurable set $E\subset\RNu$, consider the function
\begin{equation}\label{defUE}
U_E(z):=\int_E{\Gsb(\zeta^{-1}\circ z)\,{\rm d}\zeta}, \qquad z\in\RNu.
\end{equation}
In order for $U_E$ to be well-defined, we must impose a bound on the parameter $s$, as shown in the following lemma.
\begin{lemma}\label{lemdd}
Fix $s,\beta>0$ and assume $s<1+\dfrac{2}{Q}$. Then, for any $E \subseteq \RN\times [T_1,T_2], \ T_1<T_2$, there exists a constant $C = C(T_1,T_2,s,\beta) > 0$ such that
\begin{equation}\label{boundint}
U_E(z) \leq C\qquad\mbox{ for all }z\in\RNu.
\end{equation}
Moreover, for all Lebesgue-measurable sets $E\subset\RNu$ and for all $r>0$, we have
\begin{equation}\label{maxr}
\sup_{\RNu}{U_{\left(\delta_r E\right)}}=r^{Q+2-sQ}\sup_{\RNu}{U_{E}}.
\end{equation}
\end{lemma}
\begin{proof}
To prove \eqref{boundint} it suffices to show that
$$U(z):=\int_{\RN\times [T_1,T_2]}{\Gsb(\zeta^{-1}\circ z)\,{\rm d}\zeta}\leq C \qquad\mbox{for all } z\in\RNu.$$
Fix then $z=(x,t)\in\RNu$. Suppose $t> T_1$, since the other possibility is trivial. Note that $U(z)$ can be written as
$$\int_{\RN\times [T_1,\min{\{t,T_2\}})}{\frac{1}{(t-\tau)^{s\frac{Q}{2}}  }
\exp\left(  -\frac{1}{4\beta}\left\langle C_0^{-1}(  1)D_{\frac{1}{\sqrt{t-\tau}}} \left(x-E(t-\tau)\xi\right),D_{\frac{1}{\sqrt{t-\tau}}} \left(x-E(t-\tau)\xi\right)\right\rangle
\right)\,{\rm d}\xi{\rm d}\tau}.$$
By performing the change of variables $\xi'=\xi-E(\tau-t)x$ and using the commutation property \eqref{commu}, we get
\begin{eqnarray*}
&&U(z) = \int_{\RN\times [T_1,\min{\{t,T_2\}})}{\frac{1}{(t-\tau)^{s\frac{Q}{2}}  }
\exp\left(  -\frac{1}{4\beta}\left\langle C_0^{-1}(  1)D_{\frac{1}{\sqrt{t-\tau}}} E(t-\tau)\xi',D_{\frac{1}{\sqrt{t-\tau}}} E(t-\tau)\xi'\right\rangle
\right)\,{\rm d}\xi'{\rm d}\tau}\\
&=& \int_{\RN\times [T_1,\min{\{t,T_2\}})}{\frac{1}{(t-\tau)^{s\frac{Q}{2}}  }
\exp\left(  -\frac{1}{4\beta}\left\langle E^T(1)C_0^{-1}(  1)E(1)D_{\frac{1}{\sqrt{t-\tau}}} \xi',D_{\frac{1}{\sqrt{t-\tau}}} \xi'\right\rangle
\right)\,{\rm d}\xi'{\rm d}\tau}.
\end{eqnarray*}
We can now change $\xi=D_{\frac{1}{\sqrt{t-\tau}}} \xi'$ and get that the last integral is equal to
$$\left(\int_{\RN}{\exp\left(  -\frac{1}{4\beta}\left\langle E^T(1)C_0^{-1}(  1)E(1)\xi,\xi\right\rangle
\right)\,{\rm d}\xi}\right)\left(\int_{[T_1,\min{\{t,T_2\}})}{\frac{(t-\tau)^{\frac{Q}{2}}}{(t-\tau)^{s\frac{Q}{2}}}\,{\rm d}}\tau\right).$$
The second integral is finite if $s<1+\frac{2}{Q}$. The first integral can be easily bounded, and it can even be computed explicitly. Indeed, it follows from \eqref{fondfrozen} that $\int_{\RN}\Gamma_0(\zeta^{-1}\circ z)\,{\rm d}\xi=1$ for every $z$ with $t>\tau$. By choosing $x=0$ and $t=\tau+1$, we infer that $\int_{\RN}c_0\exp\left(  -\frac{1}{4}\left\langle E^T(1)C_0^{-1}(  1)E(1)\xi,\xi\right\rangle
\right)\,{\rm d}\xi=1$. This implies, by \eqref{azero},
$$\int_{\RN}{\exp\left(  -\frac{1}{4\beta}\left\langle E^T(1)C_0^{-1}(  1)E(1)\xi,\xi\right\rangle
\right)\,{\rm d}\xi}=\frac{1}{c_0\beta^{\frac{N}{2}}}\leq \left(\frac{4\pi\L}{\beta}\right)^{\frac{N}{2}}\sqrt{\deter(C(1))}.$$
This proves \eqref{boundint}. The proof of \eqref{maxr} follows by homogeneity and the properties of the group automorphisms $\delta_r$. In fact, for $z\in\RNu$ and $r>0$, we have
$$U_{\left(\delta_r E\right)}(\delta_r z)=r^{Q+2}\int_{E}{\Gsb((\delta_r \zeta)^{-1}\circ (\delta_r z))\,{\rm d}\zeta}=r^{Q+2-sQ}\int_{E}{\Gsb(\zeta^{-1}\circ z)\,{\rm d}\zeta}=r^{Q+2-sQ} U_E(z).$$
\end{proof}

In the remainder of this section, we will determine conditions on the parameters $s, \beta$ that are necessary for $U_E$ to be a subsolution for the class of variable coefficient operators \eqref{varintroL}. To this end, we compute $\elle_A \Gsb$. Recalling the definition of $\Gsb$ \eqref{defbar}, for fixed $s,\beta>0$ and for all $\O\ni z\neq 0$, we can compute
\begin{eqnarray}
\nabla\Gsb(z)&=&-\frac{1}{2\beta}\Gsb(z)\Cmu x\label{nabG}\\
D^2\Gsb(z)&=&\frac{1}{2\beta}\Gsb(z)\left(-\Cmu +\frac{1}{2\beta}\left(\Cmu x\right)\otimes\left(\Cmu x\right)\right)\label{HG}.
\end{eqnarray}
To compute the $t$-derivative, we use the following identities for invertible matrices $M(t)$:
\begin{itemize}
\item[$\cdot$] $(M^{-1}(t))'=-M^{-1}(t)M'(t)M^{-1}(t)$;
\item[$\cdot$] $(\deter(M(t)))'=\trace(M'(t)M^{-1}(t))\,\deter(M(t))$.
\end{itemize}
This yields
\begin{eqnarray*}
\de_t\Gsb(z)&=&\Gsb(z)\left(-\frac{s}{2}\frac{(\deter(C_0(t)))'}{\deter(C_0(t))}+\frac{1}{4\beta}\left\langle C_0'(t)\Cmu x,\Cmu x\right\rangle \right)=\\
&=&\Gsb(z)\left(-\frac{s}{2}\trace\left(C_0'(t)\Cmu\right) +\frac{1}{4\beta}\left\langle C_0'(t)\Cmu x,\Cmu x\right\rangle\right).
\end{eqnarray*}
We have by definition $C_0'(t)=E(t)A_0E^T(t)$. On the other hand, the following identity also holds
$$E(t)A_0E^T(t) = A_0-B^TC_0(t)-C_0(t)B\qquad\forall\,t>0.$$
To see this, note that the r.h.s. and the l.h.s. agree at $t=0$ and they have the same derivative. Consequently,
\begin{equation}\label{odecon}
C_0'(t)=A_0-B^TC_0(t)-C_0(t)B\qquad\forall\,t>0.
\end{equation}
Multiplying by $\Cmu$ and taking the trace (recall $\trace(B)=0$), we get
\begin{equation}\label{odecons}
\trace\left(C_0'(t)\Cmu\right) =\trace\left(A_0\Cmu\right).
\end{equation}
It follows from \eqref{odecon} and \eqref{odecons} that
\begin{equation}\label{tderG}
\de_t\Gsb(z)=\frac{1}{2\beta}\Gsb(z)\left(-s\beta \,\trace\left(A_0\Cmu\right)+\frac{1}{2}\left\langle A_0\Cmu x,\Cmu x\right\rangle - \left\langle B\Cmu x, x\right\rangle \right).
\end{equation}
Recalling the definition of $\elle_A$ and using \eqref{nabG}, \eqref{HG} and \eqref{tderG}, we conclude
\begin{eqnarray}\label{primadihip}
\elle_A\Gsb(z)=\frac{1}{2\beta}\Gsb(z)&&\hspace{-0.6cm}\left(s\beta \,\trace\left(A_0\Cmu\right) -\trace\left(A(z)\Cmu\right)+\vphantom{\frac{1}{2}}\right.\\
&&\left. +\frac{1}{2}\left(\frac{1}{\beta}\left\langle A(z)\Cmu x,\Cmu x\right\rangle-\left\langle A_0\Cmu x,\Cmu x\right\rangle\right)\right).\nonumber
\end{eqnarray}

Using \eqref{primadihip}, we now show that the parameters $s, \beta$ can be chosen appropriately under each of the hypotheses (H1) and (H2) so that $U_E$ is a subsolution for $\elle_A$ outside $E$.

\subsection{Cordes-Landis Condition \hyperref[H1]{H1}}\label{cordeslandiscondition}

Consider the kernel $\Gsb$ with the choice
$$A_0=I_0,$$
where $I_0$ is defined in \eqref{matrixexponential}. Note that with this choice of $A_0$ we have $C_0 = C$, where $C$ is defined in \eqref{Kalma0}. Fix also 
\begin{equation}\label{defh1bs}
\beta=\l\qquad\mbox{ and }\qquad s=\frac{\L}{\l}.
\end{equation}
Having fixed these quantities, we can prove the following
\begin{lemma}\label{lemuu}
Suppose $E \subset \RNu$ is bounded and \eqref{defh1bs} holds. Then the function $U_E$ is continuous in $\RNu$, $C^2$ outside $\overline{E}$, and for all $A$ satisfying \hyperref[H1]{H1}, we have
$$\elle_A U_E(z)\geq 0 \quad\mbox{ for }z\in\O\smallsetminus \overline{E}.$$
\end{lemma}
\begin{proof}
The choice of $s$ and the hypothesis \hyperref[H1]{H1} allow us to invoke Lemma \ref{lemdd} and conclude that $U_E<+\infty$. Moreover, with such choices, $U_E$ is continuous by the dominated convergence theorem and smooth away from the poles in $E$. 
Let us now prove the $\elle_A$-subsolution property. By \eqref{defh1bs}, and using \eqref{ellipticityofA} in \eqref{primadihip}, we have
\begin{eqnarray*}
\elle_A\Gsb(z)=\frac{1}{2\beta}\Gsb(z)&&\hspace{-0.6cm}\left(s\beta \,\trace\left(I_0\CmuI\right) -\trace\left(A(z)\CmuI\right)+\vphantom{\frac{1}{2}}\right.\\
&&\left. +\frac{1}{2}\left(\frac{1}{\beta}\left\langle A(z)\CmuI x,\CmuI x\right\rangle-\left\langle I_0\CmuI x,\CmuI x\right\rangle\right)\right) \\
\geq\frac{1}{2\beta}\Gsb(z)&&\hspace{-0.6cm}\left((s\beta-\L) \,\trace\left(I_0\CmuI\right) +\frac{1}{2}\left(\frac{\l}{\beta}-1\right)\left\langle I_0\CmuI x,\CmuI x\right\rangle\right)=0.
\end{eqnarray*}
To complete the proof, we have only to remember that the vector fields $\de_{x_1}, \ldots, \de_{x_N}$ and $\left\langle x, B\nabla\right\rangle - \partial_t$ are left-invariant with respect to the group law $\circ$. Thus, the function $\Gsb(\zeta^{-1}\circ \cdot)$ is $\elle_A$-subharmonic for any $\zeta \in \mathbb{R}^{N+1}$, and the lemma follows.
\end{proof}

\subsection{Uniform Continuity Assumption \hyperref[H2]{H2}}\label{continuityassumption}

We make precise here the notion of uniform continuity of the coefficients as stipulated in condition \hyperref[H2]{H2}.

\begin{definition}\label{modulusofcontinuity}
Let $\omega : [0, 1) \rightarrow [0,1)$ be a non-decreasing function satisfying $\lim\limits_{s \rightarrow 0^+} \omega(s) = \omega(0) = 0$. We say the matrix $A(\cdot)$ admits a uniform modulus of continuity $\omega$ if
$$\sup\left\{\left\|\mathbb{A}(z)-\mathbb{A}(z_0)\right\|\,:\,z \in\O\cap  \Q_\e^{-\e^2,\e^2}(z_0) \right\} \leq \omega(\epsilon)\quad\mbox{ for all }0<\epsilon <1\,\, \mbox{ and for all }z_0 \in \Omega.$$
\end{definition}

Assume now that $A(\cdot)$ admits a uniform modulus of continuity $\omega$. Fix any $z_0\in\O$ and choose
$$A_0=A(z_0).$$
Define
$$C_{z_0}(t):=C_0(t)=\int_{0}^{t} E(\sigma)A(z_0)E^T(\sigma)\,d\sigma.$$
Let
\begin{equation}\label{pick}
s=1+s_0,\quad\beta=\frac{1}{1+\frac{s_0}{2}}=\frac{2}{2+s_0}\qquad\mbox{for some }s_0>0 \text{ to be determined}.
\end{equation}
Let $\Gsb$ be the kernel corresponding to the above choices. We want the associated potential $U_E$ to be a subsolution in a neighborhood of $z_0$ for $0<s_0<\frac{2}{Q}$. To do this, we exploit the continuity of $A(\cdot)$.
\begin{lemma}\label{barriersforH2} Suppose $E \subset \RNu$ is bounded and \eqref{pick} holds. Then for every $0<s_0<\frac{2}{Q}$ there exists $0<\e_0<1$ depending on $s_0$ and $\omega(\cdot)$ such that
$$\elle_A U_E(z) \geq 0 \qquad\mbox{ for all } z \in \O \cap \Q_{\e_0}^{-\e_0^2,\e_0^2}(z_0)\smallsetminus\overline{E}.$$
\end{lemma}
\begin{proof}
It suffices to show that for all $\zeta\in E$, we have
$$\left(\elle_A \Gsb(\zeta^{-1}\circ\cdot)\right)(z)\geq 0\qquad\mbox{ for all }z\in\O\cap  \Q_{\e_0}^{-\e_0^2,\e_0^2}(z_0)\smallsetminus\{\zeta\}.$$
Fix $0<s_0<\frac{2}{Q}$. By Definition \ref{modulusofcontinuity}, there exists $\e_0>0$ such that
$$\left\|\mathbb{A}(z)-\mathbb{A}(z_0)\right\|\leq\omega(\e_0)\leq \frac{s_0}{2+s_0}\l\qquad\mbox{ for all }z\in\O\cap \Q_{\e_0}^{-\e_0^2,\e_0^2}(z_0).$$
In particular, this implies
\begin{equation}\label{conscont}
-\frac{s_0}{2}\l I_0 \leq A(z)-A(z_0)\leq \frac{s_0}{2+s_0}\l I_0\qquad\mbox{ for all }z\in \O\cap \Q_{\e_0}^{-\e_0^2,\e_0^2}(z_0).
\end{equation}
Let us now fix $\zeta\in\RNu$. Arguing as in \eqref{primadihip}, we have for any $z\neq\zeta$
\begin{eqnarray*}
\left(\elle_A \Gsb(\zeta^{-1}\circ\cdot)\right)(z)&=& \frac{1}{2\beta}\Gsb(\zeta^{-1}\circ z)\left(s\beta \,\trace\left(A(z_0)C^{-1}_{z_0}(t-\tau)\right) -\trace\left(A(z)C^{-1}_{z_0}(t-\tau)\right)+\vphantom{\frac{1}{2}}\right.\\
&&+\frac{1}{2}\left(\frac{1}{\beta}\left\langle A(z)C^{-1}_{z_0}(t-\tau) (x-E(t-\tau)\xi),C^{-1}_{z_0}(t-\tau) (x-E(t-\tau)\xi)\right\rangle+\right.\\
&&-\left.\left.\vphantom{\frac{1}{2}}\left\langle A(z_0)C^{-1}_{z_0}(t-\tau) (x-E(t-\tau)\xi),C^{-1}_{z_0}(t-\tau) (x-E(t-\tau)\xi)\right\rangle\right)\right).
\end{eqnarray*}
Let us bound from below separately the trace-terms and the quadratic-terms. Consider any $z\in \O\cap \Q_{\e_0}^{-\e_0^2,\e_0^2}(z_0)$. With our choice of $\beta$ in \eqref{pick} we have $\frac{1}{\beta}=1+\frac{s_0}{2}$. Thus, using \eqref{ellipticityofA} and \eqref{conscont}, we get
$$M_1(z):=\frac{1}{\beta}A(z)-A(z_0)=A(z)-A(z_0)+\frac{s_0}{2}A(z)\geq A(z)-A(z_0)+\frac{s_0}{2}\l I_0\geq 0$$
which implies
$$\left\langle M_1(z) C^{-1}_{z_0}(t-\tau) (x-E(t-\tau)\xi),C^{-1}_{z_0}(t-\tau) (x-E(t-\tau)\xi)\right\rangle\geq 0.$$
On the other hand, by \eqref{pick} we have $s\beta=1+\frac{s_0}{2+s_0}$. Again, by \eqref{ellipticityofA} and \eqref{conscont}, we get
$$M_2(z):=s\beta A(z_0)-A(z)=A(z_0)-A(z)+\frac{s_0}{2+s_0} A(z_0)\geq A(z_0)-A(z)+\frac{s_0}{2+s_0}\l I_0\geq 0$$
which implies
$$\trace\left(M_2(z) C^{-1}_{z_0}(t-\tau)\right)\geq 0.$$
Hence
\begin{eqnarray*}
&&\left(\elle_A \Gsb(\zeta^{-1}\circ\cdot)\right)(z)= \frac{1}{2\beta}\Gsb(\zeta^{-1}\circ z)\left(\trace\left(M_2(z)C^{-1}_{z_0}(t-\tau)\right) +\vphantom{\frac{1}{2}}\right.\\
&&\left.+\frac{1}{2}\left\langle M_1(z)C^{-1}_{z_0}(t-\tau) (x-E(t-\tau)\xi),C^{-1}_{z_0}(t-\tau) (x-E(t-\tau)\xi)\right\rangle\right)\geq 0
\end{eqnarray*}
for every $z\in \O\cap \Q_{\e_0}^{-\e_0^2,\e_0^2}(z_0)\smallsetminus\{\zeta\}$.
\end{proof}

\begin{remark}
In lieu of the hypothesis \hyperref[H1]{H1}, we could have assumed the existence of a fixed matrix $A_0$ of the form \eqref{azero} such that
\begin{equation}\label{generalcordes}
    \tilde{\l} A_0\leq A(z) \leq \tilde{\L} A_0 \quad\mbox{ with }\quad \frac{\tilde{\L}}{\tilde{\l}}<1+\frac{2}{Q}.
\end{equation}
The proof of Lemma \ref{lemuu} can then be carried out in exactly the same way. This is in contrast with operators in groups of Heisenberg type considered in \cite{AGT} and \cite{Tralli-Critical-Density}, where it is not clear how to establish the analogue of Lemma \ref{lemuu} under the more general condition \eqref{generalcordes} without making additional structural assumptions on the coefficient matrix. A similar obstruction arises when attempting to prove the analogue of Lemma \ref{barriersforH2} (see \cite[Section 3]{AGT}).
\end{remark}

\section{Growth Lemma and Applications}\label{sec5}

In this section, we establish the Landis growth lemma for the operators $\mathcal{L}_A$ under the hypotheses \hyperref[H1]{H1} and \hyperref[H2]{H2}. It is well known that such growth lemmas are the starting point for proving oscillation decay, H\"older continuity and Harnack's inequality for solutions.

Let us recall the definitions of the cylinder-like sets considered in Lemma \ref{upperboundlemma} and \ref{lowerboundlemma}
$$\Q^1_r(z_0)=\Q_{K r}^{-b_B r^2, 0}(z_0) \qquad \Q^2_r(z_0)=\Q_{\sigma_0 r}^{-\frac{1}{4}b_B r^2, 0}(z_0) \qquad \Q^3_r(z_0)=\Q_{\sigma_0 r}^{-b_B r^2, -\frac{1}{2}b_B r^2}(z_0).$$
In Lemma \ref{upperboundlemma}, it was shown that there exists a constant $K_1 > \sigma_0$ depending only on the structure such that for all $K > K_1$, we have the upper bound on $\Gsb$ given in  \eqref{upperbound}. We can choose the constant $K$ large enough so that the bound in \eqref{lowerbound} is greater than the bound in \eqref{upperbound}. To this end, we fix $K > 0$ satisfying
\begin{equation}\label{choiceofK}
    K^2 > \max \left\{\frac{c_2}{c_1 b_B^{2n}}, K_1^2 \right\}.
\end{equation}

\begin{theorem}\label{growthlemma} (``Growth Lemma'' under \hyperref[H1]{H1}) Let $z_0 \in \O$, and consider an open set $D \subseteq \Q^1_r(z_0)\subset \O$ such that $D\cap \Q^2_r(z_0)\neq \emptyset$. Suppose $u\in C^2(D) \cap C\left(\overline{D}\right)$ is nonnegative in $D$, vanishes on $\de D\cap \Q^1_r(z_0)$, and satisfies $\elle_A u \geq 0$ in $D$. Assume, moreover, that the Cordes-Landis condition \hyperref[H1]{H1} holds for the operator $\elle_A$. Then there exists a structural constant $\eta > 0$ such that
$$\sup\limits_D u \geq \left(1 + \eta \dfrac{|\Q^3_r(z_0) \backslash D|}{|\Q^3_r(z_0)|} \right) \sup\limits_{D \cap \Q^2_r(z_0)} u.$$
\end{theorem}
\begin{proof} 
By translation invariance of the class of operators under consideration, we may assume $z_0 = 0\in\O$. We may also assume $u$ is non-trivial, and so $D$ has limit points on $\de_p\Q^1_r(z_0)=S^1_r \cup \left(B_{Kr}(0) \times \left\{-b_Br^2 \right\} \right)$ by the weak maximum principle \eqref{MP} (recall from Section \ref{secest} that $S^1_r=\de B_{Kr}(0)\times [-b_Br^2 ,0]$). Let $E = \Q^3_r \backslash D$ and consider the function $U_E$ defined in \eqref{defUE} with the choice of $\Gsb$ as in subsection \ref{cordeslandiscondition} (recall \eqref{defh1bs}). If we call $C$ the positive structural constant given in \eqref{boundint} such that $\sup U_{\Q^3_1}\leq C$, then we have by \eqref{maxr}
\begin{equation}\label{maxrE}
\sup{U_E}\leq \sup{U_{\Q^3_r}}\leq C r^{Q+2-sQ}.
\end{equation} 
Moreover, by the bounds \eqref{upperbound} and \eqref{lowerbound}, we have
\begin{equation}\label{boundsonsides}
\sup_{S^1_r} U_E \leq \frac{e^{-\mu_2}}{(b_B r^2)^{s\frac{Q}{2}}}|E|\qquad\mbox{with }\mu_2:=\frac{c_1 K^2}{b_B} ,
\end{equation}
and
\begin{equation}\label{boundsontop}
\inf_{\Q^2_r} U_E \geq \frac{e^{-\mu_1}}{(b_B r^2)^{s\frac{Q}{2}}}|E|\qquad\mbox{with }\mu_1:= \frac{c_2}{b_B^{2n+1}}.
\end{equation}
Consider the auxiliary function
$$v(z) = \sup_D u \left(1 - \frac{U_E(z)}{C r^{Q+2-sQ}} + \frac{|E|}{C r^{Q+2}}\frac{e^{-\mu_2}}{(b_B)^{s\frac{Q}{2}} }\right).$$
Then $v$ is non-negative everywhere by \eqref{maxrE}. Since $\elle_A U_E(z)\geq 0$ for all $z\in\O\smallsetminus \overline{E}$ by Lemma \ref{lemuu}, we have $\elle_A v \leq 0\leq \elle_A u$ in $D$. We now want to compare $v$ and $u$ on the portion of $\partial D$ required to apply the weak maximum principle. For this purpose, we define the sets $\gamma:=\de D\cap \Q^1_r$, $\gamma_1 := \overline{D} \cap (B_{Kr}(0) \times \left\{-b_Br^2 \right\})$, and $\gamma_2 := \overline{D} \cap S^1_r$. Since $u = 0$ on $\gamma$, $v \geq u$ on $\gamma$. Recall also that $U_E$ is a continuous function. Since $U_E(z) = 0$ for $z \in B_{Kr}(0) \times \left\{-b_B r^2 \right\}$, we then have $v(z) \geq \sup_D u \geq u(z)$ for all $z \in \gamma_1$. Finally, for $z \in \gamma_2$, we have by \eqref{boundsonsides}
$$    v(z)  \geq \sup_D u \left(1 - \dfrac{\sup_{S^1_r} U_E}{C r^{Q+2-sQ}} + \frac{|E|}{C r^{Q+2}}\frac{e^{-\mu_2}}{(b_B)^{s\frac{Q}{2}}}\right)\geq \sup_D u. $$
Thus $v \geq u$ on $\gamma_2$. By the weak maximum principle, it follows that $v \geq u$ in $D$. Hence, for $z \in D \cap \Q^2_r\neq \emptyset$, we have by \eqref{boundsontop}
\begin{align*}
    u(z) & \leq \sup_D u \left(1 - \frac{\inf_{\Q^2_r} U_E}{C r^{Q+2-sQ}} + \frac{|E|}{C r^{Q+2}}\frac{e^{-\mu_2}}{(b_B)^{s\frac{Q}{2}}}\right) 
     \leq \sup_D u \left(1 - \frac{|E|}{C r^{Q+2}}\frac{e^{-\mu_1}}{ (b_B)^{s\frac{Q}{2}}} + \frac{|E|}{C r^{Q+2}}\frac{e^{-\mu_2}}{(b_B)^{s\frac{Q}{2}}}\right) \\
    & = \sup_D u \left(1 - \frac{e^{-\mu_1} - e^{-\mu_2}}{C \, (b_B )^{s\frac{Q}{2}}} \frac{|E|}{r^{Q+2}} \right).
\end{align*}
By \eqref{choiceofK}, we have $\mu_1 < \mu_2$. Hence, we can define $\bar{\eta} := \dfrac{e^{-\mu_1} - e^{-\mu_2}}{C\,b_B^{sQ/2}} > 0$ and we can write
$$\bar{\eta}\, \frac{|E|}{r^{Q+2}}=\bar{\eta}\,\frac{|\Q^3_r \backslash D|}{r^{Q+2}}=\left(\bar{\eta}|\Q^3_1|\right)\frac{|\Q^3_r \backslash D|}{|\Q^3_r|}=:\eta\,\frac{|\Q^3_r \backslash D|}{|\Q^3_r|}.$$
This completes the proof.
\end{proof} 

The proof of Theorem \ref{growthlemma} also allows us to obtain the following version of the growth lemma under the condition \hyperref[H2]{H2}. More specifically, we assume the continuity assumption \hyperref[H2]{H2} holds for the operator $\elle_A$ and we let $\epsilon_0$ be the constant from Lemma \ref{barriersforH2} corresponding to the choice $s_0 := \frac{1}{Q}$, which we fix from here onwards.

\begin{theorem}\label{growthlemmaH2} (``Growth Lemma'' under \hyperref[H2]{H2})  Let $z_0 \in \O$, and suppose $0 < r \leq \frac{\epsilon_0}{K}$. Consider an open set $D \subseteq \Q^1_r(z_0)\subset \O$ such that $D\cap \Q^2_r(z_0)\neq \emptyset$. Suppose $u\in C^2(D) \cap C\left(\overline{D}\right)$ is nonnegative in $D$, vanishes on $\de D\cap \Q^1_r(z_0)$, and satisfies $\elle_A u \geq 0$ in $D$. Assume, moreover, that the continuity condition \hyperref[H2]{H2} holds for the operator $\elle_A$. Then there exists a constant $\eta > 0$ such that
$$\sup\limits_D u \geq \left(1 + \eta \dfrac{|\Q^3_r(z_0) \backslash D|}{|\Q^3_r(z_0)|} \right) \sup\limits_{D \cap \Q^2_r(z_0)} u.$$
\end{theorem}

\begin{proof} The proof is essentially that of Theorem \ref{growthlemma}. The only modification is that the function $U_E$ is now constructed with the choice of $\Gsb$ as in subsection \ref{continuityassumption}, see \eqref{pick}. By Lemma \ref{barriersforH2}, $U_E$ is a subsolution only inside a cylinder $\Q_{\e_0}^{-\e_0^2,\e_0^2}(z_0)$ of size $\e_0$ depending on the modulus of continuity $\omega$ for the coefficients $A(\cdot)$. The assumption $0 < r \leq \frac{\e_0}{K}$ (recall that $K\geq 1 \geq\sqrt{b_B}$) yields $D \subseteq \Q^1_r(z_0)\subset \O\cap \Q_{\e_0}^{-\e_0^2,\e_0^2}(z_0)$.
\end{proof}

We provide an immediate application of Theorem \ref{growthlemma} and Theorem \ref{growthlemmaH2} by showing oscillation decay and H\"older continuity of solutions to $\elle_A u = 0$. Recall that the oscillation of a function $u$ over a set $E$ is defined to be $\mathrm{osc}_{E} u := \sup_E u - \inf_E u$.

\begin{corollary}\label{oscillationdecayandHolder}(Oscillation Decay and Local H\"older Continuity under \hyperref[H1]{H1}) Suppose the operator $\elle_A$ satisfies the hypothesis \hyperref[H1]{H1}. There exists a structural constant $\theta > 1$ such that if $u$ is a $C^2$ solution of $\elle_A u = 0$ in an open set $D\subseteq \O$ and $\Q^1_r(z_0)\Subset D$, then 
\begin{equation}\label{oscdecay}
\mathrm{osc}_{\Q^1_{r}(z_0)} u  \geq \left(1 + \frac{\eta}{4}\right) \ \mathrm{osc}_{\Q^1_{r/\theta}(z_0)} u.
\end{equation}
Consequently, there exists a structural constant $\alpha$ and, for any $\rho>0$, a positive constant $C_\rho$ such that
\begin{equation}\label{BHolderestimate}
|u(z) - u(\zeta)|\leq C_\rho\left\|\zeta^{-1}\circ z\right\|^{\alpha}_B \left\|u\right\|_{L^{\infty}(D)}\quad\mbox{ for all }z,\zeta\mbox{ such that }\Q^1_\rho(z),\Q^1_\rho(\zeta)\Subset D.
\end{equation}
\end{corollary}

\begin{proof} We first prove \eqref{oscdecay}. Let $\theta := \frac{K}{\sigma_0}$. Recalling that $\theta\geq 2$, we have $\Q^1_{r/\theta}(z_0) \subseteq \Q^2_r(z_0)$ and so $\mathrm{osc}_{\Q^1_{r/\theta}(z_0)} u \leq \mathrm{osc}_{\Q^2_{r}(z_0)} u$. Consider the function
$$v = 2 u - \left(\sup\limits_{\Q^2_r(z_0)} u + \inf\limits_{\Q^2_r(z_0)} u\right),$$
and let $D^+ := \Q^1_r(z_0) \cap \left\{v > 0 \right\}$. We may assume, without loss of generality, that $|\Q^3_r(z_0) \backslash D^+| \geq \frac{1}{2}|\Q^3_r(z_0)|$; otherwise, consider $-v$ instead of $v$. In addition, it suffices to assume $D^+ \cap \Q^2_r(z_0) \neq \emptyset$, as otherwise the function $v$ is identically constant in $\Q^2_r(z_0)$ and $\mathrm{osc}_{\Q^1_{r/\theta}(z_0)} u=0$. Applying Theorem \ref{growthlemma} to the function $v$ with $D = D^+$, we obtain
\begin{eqnarray*}
2\mathrm{osc}_{\Q^1_r(z_0)} u - \mathrm{osc}_{\Q^2_r(z_0)} u&\geq& \sup\limits_{\Q^1_r(z_0)} v \geq \left(1 + \eta \dfrac{|\Q^3_r(z_0) \backslash D^+|}{|\Q^3_r(z_0)|} \right) \sup\limits_{D^+ \cap \Q^2_r(z_0)} v \geq \left(1 + \dfrac{\eta}{2} \right) \sup\limits_{\Q^2_r(z_0)} v \\
&=& \left(1 + \dfrac{\eta}{2} \right)\mathrm{osc}_{\Q^2_r(z_0)} u.
\end{eqnarray*}
which implies
\begin{equation}\label{oscdec}
\mathrm{osc}_{\Q^1_{r}(z_0)} u  \geq P \,\mathrm{osc}_{\Q^2_{r}(z_0)} u
\geq P \,\mathrm{osc}_{\Q^1_{r/\theta}(z_0)} u\quad\,\mbox{with }\,\,P := 1 + \dfrac{\eta}{4}.
\end{equation}

To prove the estimate \eqref{BHolderestimate}, fix $\rho>0$ and let $z,\zeta$ be arbitrary points in $D$ such that $\Q^1_\rho(z),\Q^1_\rho(\zeta)\Subset D$. With no loss of generality, we may assume $t\leq \tau$. We have two cases: either $z\in \Q^1_\rho(\zeta)$ or $z\notin \Q^1_\rho(\zeta)$.\\ 
If $z\in \Q^1_\rho(\zeta)$, choose $m_0 \in \mathbb{N}\cup\{0\}$ such that $z \in \Q^1_{\frac{\rho}{\theta^{m_0}}}(\zeta)$ and $z \notin \Q^1_{\frac{\rho}{\theta^{m_0+1}}}(\zeta)$. Hence, $\left\|\zeta^{-1}\circ z\right\|_B \geq \min\left\{K, \sqrt{b_B} \right\} \frac{\rho}{\theta^{m_0+1}}=\sqrt{b_B}\frac{\rho}{\theta^{m_0+1}}$. Applying \eqref{oscdec} recursively, we obtain
$$\mathrm{osc}_{\Q^1_{\frac{\rho}{\theta^{m_0}}}(\zeta)} u \leq \frac{1}{P^{m_0}} \mathrm{osc}_{\Q^1_\rho} u \leq \frac{2 P \left\|u\right\|_{L^{\infty}(D)}}{P^{m_0+1}}.$$
Writing $P^{m_0+1} = (\theta^{\log_{\theta} P})^{m_0 + 1} = (\theta^{m_0+1})^{\log_{\theta} P}$ and letting $\alpha:=\log_{\theta} P$, we get
$$|u(z) - u(\zeta)| \leq \text{osc}_{\Q^1_{\frac{\rho}{\theta^{m_0}}}(\zeta)} u \leq \frac{2 P \left\|u\right\|_{L^{\infty}(D)}}{b_B^{\frac{\alpha}{2}}\rho^{\alpha}}\left(\sqrt{b_B}\frac{\rho}{\theta^{m_0+1}}
 \right)^{\alpha} \leq  \frac{2 P \left\|u\right\|_{L^{\infty}(D)}}{b_B^{\frac{\alpha}{2}}\rho^{\alpha}} \left\|\zeta^{-1}\circ z\right\|^{\alpha}_B.$$
 On the other hand, if $z\notin \Q^1_\rho(\zeta)$ we simply have $\left\|\zeta^{-1}\circ z\right\|_B \geq \sqrt{b_B}\rho$ and then
 $$|u(z) - u(\zeta)| \leq 2 \left\|u\right\|_{L^{\infty}(D)}\leq \frac{2 \left\|u\right\|_{L^{\infty}(D)}}{b_B^{\frac{\alpha}{2}}\rho^{\alpha}} \left\|\zeta^{-1}\circ z\right\|^{\alpha}_B.$$
Combining the two possibilities, we obtain the desired estimate \eqref{BHolderestimate} with the choice $C_\rho=2 P b_B^{-\frac{\alpha}{2}}\rho^{-\alpha}$.
\end{proof}

In order to establish the corresponding version of Corollary \ref{oscillationdecayandHolder} under the hypothesis \hyperref[H2]{H2}, we notice that the proof can be carried out in the same manner simply by considering $\Q^1_r(z_0)\subset D$ with $0<r\leq\frac{\e_0}{K}$. The constant $C_{\rho}$ will now depend additionally on $\epsilon_0$; however, the constant $\alpha$ remains independent of $\epsilon_0$.

\begin{remark} 
The regularity estimate \eqref{BHolderestimate} is equivalent to local H\"older continuity in the standard sense. One can see this by comparing the $|\cdot|_B$-norm with the Euclidean norm as in \eqref{triangleB}--\eqref{benne} (see also \cite[Definition 1.2 and Proposition 2.1]{Poli}). \end{remark}

\section{Harnack Inequality}\label{harn}

In this final section, we prove the Harnack inequality for non-negative solutions to $\elle_A u = 0$ using the growth lemma. We follow closely the approach outlined by Landis in \cite[Lemma 9.1 and Theorem 9.1]{Landis} and make a number of necessary modifications to adapt his proof to our setting.

\begin{lemma}\label{incylinder}
There exist structural constants $C_1,C_2>0$ such that:
\begin{itemize}
\item[(i)] for any $R>0$ and any $0\leq\delta_1<\delta_2\leq\frac{1}{2}$, if $\rho\leq C_1 R (\delta_2-\delta_1)^{n+\frac{1}{2}}$ then
$$\Q_{K \rho}^{-b_B\rho^2,0}(z_0)\subseteq \Q_{R\left(\frac{1}{2}+\delta_2\right)}^{-b_B R^2\left(\frac{1}{2}+\delta_2\right),0}\qquad \forall\,z_0\in \overline{\Q_{R\left(\frac{1}{2}+\delta_1\right)}^{-b_B R^2\left(\frac{1}{2}+\delta_1\right),0}}; $$

\item[(ii)] for any $R>0$ and any $0\leq\delta_1<\delta_2\leq1$, if $\rho\leq C_2 R (\delta_2-\delta_1)^{n+\frac{1}{2}}$ then
$$\Q_{\rho}^{-b_B\rho^2,0}(z_0)\subseteq \Q_{R\frac{\sigma_0}{2}\left(1+\delta_2\right)}^{-\frac{b_B}{4} R^2\left(3+\delta^2_2\right),-\frac{b_B}{2} R^2}\smallsetminus \Q_{R\frac{\sigma_0}{2}\left(1+\delta_1\right)}^{-\frac{b_B}{4} R^2\left(3+\delta^2_1\right),-\frac{b_B}{2} R^2}\qquad \forall\,z_0\in \de_p \Q_{R\frac{\sigma_0}{2}\left(1+\frac{\delta_1+\delta_2}{2}\right)}^{-\frac{b_B}{4} R^2\left(3+\frac{(\delta_1+\delta_2)^2}{4}\right),-\frac{b_B}{2} R^2}. $$
\end{itemize}
\end{lemma}
We postpone the proof of this lemma to the end of the section. For now, we use it to prove the following important consequence of the growth lemma.

\begin{lemma}\label{lemma91} Let $\bar{z} \in \O$, and consider an open set $D \subseteq \Q^{-b_B R^2,0}_R(\bar{z})\subset \O$ such that $D^+:=D\cap \Q^{-\frac{1}{2}b_B R^2,0}_{\frac{1}{2}R}(\bar{z})\neq \emptyset$. Suppose $u\in C^2(D)\cap C\left(\overline{D}\right)$, nonnegative in $D$, vanishes on $\de D\cap \Q^{-b_B R^2,0}_R(\bar{z})$, and satisfies $\elle_A u = 0$ in $D$. Assume, moreover, that the Cordes-Landis condition \hyperref[H1]{H1} holds for the operator $\elle_A$. Then, for any $M>1$, there exists $\delta > 0$ (depending on $M$ and on structural constants) such that, if $|D|\leq \delta R^{Q+2}$, we have
$$\sup\limits_D{u} \geq M \sup\limits_{D^+}{u}.$$
\end{lemma}
\begin{proof} By translation invariance we can assume $\bar{z}=0\in\O$. Let $\eta$ be the constant in Theorem \ref{growthlemma}. For any $M>1$, let $m$ be the smallest natural number such that $\left(1+\frac{\eta}{2}\right)^m>M$. For $i\in\{0,\ldots,m\}$, denote
$$\Q^{(i)}=\Q^{-\frac{1}{2}b_B R^2\left(1 + \frac{i}{m}\right),0}_{\frac{1}{2}R\left(1 + \frac{i}{m}\right)}.$$
It suffices to assume $\sup\limits_{D^+}{u}>0$, as otherwise the statement is trivial. Then, for any $i\in\{0,\ldots,m-1\}$, we can choose a point $z^i=(x^i,t^i)$ in the parabolic boundary of $\Q^{(i)}$ such that $u(z^i)=\sup_{D\cap \Q^{(i)}}{u}$. The existence of $z^i$ is guaranteed by the weak maximum principle \eqref{MP}, which implies that if there is no such $z^i$, then $0=\sup_{D\cap \Q^{(i)}}{u}\geq \sup_{D^+}{u}$. Let us now denote
$$\Q^{1,(i)}_\rho=\Q_{K \rho}^{-b_B \rho^2, 0}(z^i).$$
By recalling that $K\geq 1$ and by exploiting Lemma \ref{incylinder} (item $(i)$, with $\delta_1=\frac{i}{2m}$, $\delta_2=\frac{i+1}{2m}$) we know that 
$$\rho=\frac{C_1}{(2m)^{n+\frac{1}{2}}} R \,\,\,\,\mbox{ yields }\,\,\,\, \Q^{1,(i)}_\rho\subseteq \Q^{(i+1)}.$$
 We are going to prove the statement of the lemma with the choice
$$\delta=	\frac{1}{2}C_1^{Q+2}\frac{\left|\Q^{-b_B,-\frac{1}{2}b_B}_{\sigma_0}\right|}{(2m)^{\left(n+\frac{1}{2}\right)(Q+2)}}.$$
In fact, defining $\Q^{2,(i)}_\rho:=\Q_{\sigma_0 \rho}^{-\frac{1}{4}b_B \rho^2, 0}(z^i)$ and $\Q^{3,(i)}_\rho:=\Q_{\sigma_0 \rho}^{-b_B \rho^2, -\frac{1}{2}b_B \rho^2}(z^i)$, the assumption $|D|\leq \delta R^{Q+2}$ implies
$$|D\cap \Q^{3,(i)}_\rho|\leq \delta R^{Q+2}=\frac{1}{2} |\Q^{3,(i)}_\rho|\quad\mbox{ and hence }\quad |\Q^{3,(i)}_\rho\smallsetminus D|= |\Q^{3,(i)}_\rho| - |D\cap \Q^{3,(i)}_\rho|\geq \frac{1}{2}|\Q^{3,(i)}_\rho|.$$
Applying Theorem \ref{growthlemma} in the cylinder $\Q^{1,(i)}_\rho$, and using the inclusion $\Q^{1,(i)}_\rho\subseteq \Q^{(i+1)}$, we get
$$\sup\limits_{D\cap \Q^{(i+1)}}{u} \geq \sup\limits_{D\cap \Q^{1,(i)}_\rho}{u} \geq \left(1 + \frac{\eta}{2} \right) \sup\limits_{D \cap \Q^{2,(i)}_\rho}{u}\geq  \left(1 + \frac{\eta}{2} \right) u(z^i) = \left(1 + \frac{\eta}{2} \right) \sup\limits_{D \cap \Q^{(i)}}{u}. $$
This holds true for every $i\in\{0,\ldots,m-1\}$. Therefore, since $ \Q^{(0)}= \Q^{-\frac{1}{2}b_B R^2,0}_{\frac{1}{2}R}$ and $\Q^{(m)}=\Q^{-b_B R^2,0}_R$, we finally obtain
$$\sup\limits_D{u} \geq \left(1+\frac{\eta}{2}\right)^m \sup\limits_{D^+}{u} \geq M \sup\limits_{D^+}{u}.$$ 
\end{proof}

We are finally ready to show the proof of the Harnack inequality. We begin with the proof of Theorem \ref{harnackineqH1}
\begin{proof}[Proof of Theorem \ref{harnackineqH1}]
Assume, without loss of generality, that $z_0 = 0$ and $\sup_{\Q_r^-} u = 2$. The aim is to find a structural lower bound for $u$ on $\Q_r^+$. Let us recall the definitions of the cylinders
$$\Q^2_r=\Q_{\sigma_0 r}^{-\frac{1}{4}b_B r^2, 0} \qquad \Q^3_r=\Q_{\sigma_0 r}^{-b_B r^2, -\frac{1}{2}b_B r^2}.$$
Notice that $\Q_r^+ \subset \Q^2_r$ and $\Q_r^- \subset \Q^3_r$. Consider the set $G := \left\{z \in \Q^3_r : u(z) > 1 \right\}$. Let $\delta > 0$ be the number from Lemma \ref{lemma91} corresponding to the choice of $M = 2^{1+(n+\frac{1}{2})(Q+2)}$, and define the structural constant
$$\epsilon_0 := \left(\frac{C_2}{2^{n+\frac{1}{2}}}\right)^{Q+2} \delta,$$
where $C_2$ is the constant appearing in Lemma \ref{incylinder}, item $(ii)$. We are faced with two possibilities:
\begin{itemize}
\item[-] Case 1: $|G| \geq \epsilon_0 r^{Q+2}$, or
\item[-] Case 2: $|G| < \epsilon_0 r^{Q+2}$.
\end{itemize}
For Case 1, we consider the function $w = 1 - u$. With the intent of applying Theorem \ref{growthlemma}, we define the set $D := \left\{z \in \Q^1_r : w(z) > 0 \right\}$. We may assume $D \cap \Q^2_r \neq \emptyset$, for otherwise $u\geq 1$ in $\Q^2_r\supset \Q_r^+$. Since $u$ is non-negative, we have $w \leq 1$ in $\Q^1_r$. Furthermore, $G \subset \Q^3_r \backslash D$, and so $|\Q^3_r \backslash D| \geq \epsilon_0 r^{Q+2}$. It follows from Theorem \ref{growthlemma} applied to $w$ that
$$1 \geq \sup_{\Q^1_r} w \geq \left(1 + \eta\frac{\epsilon_0 }{|\Q^3_1|}\right) \sup_{\Q^2_r \cap D} w \geq \left(1 + \eta\frac{\epsilon_0 }{|\Q^3_1|}\right) \sup_{\Q_r^+} w.$$
Thus,
$$\inf_{\Q_r^+} u \geq \frac{\hat{C}}{1 + \hat{C}} \qquad \text{ where } \hat{C} := \eta\frac{\epsilon_0 }{|\Q^3_1|}.$$
Consequently, \eqref{harnackH1} follows when Case 1 holds.

\noindent For Case 2, we carry out an iteration procedure, which we describe in the following steps:

{\bf Step 1:} Set 
$$\Q^{(s)}:=\Q_{r\frac{\sigma_0}{2}(1+s)}^{-\frac{b_B}{4}r^2(3+s^2),-\frac{b_B}{2}r^2}, \qquad s > 0.$$
Notice that $\Q^{(0)} = \Q_r^-$, while $\Q^{(1)} = \Q^3_r$. Consider the family of sets
$$G^{(0)}_s:=G\cap(\Q^{(s)} \backslash \Q^{(0)}), \qquad 0 < s < 1.$$
Observe that 
\begin{equation}\label{valueatonehalf}
\left|G^{(0)}_{1/2}\right| \leq |G| < \epsilon_0 r^{Q+2} =  \left(C_2 r\left(\frac{1}{2} \right)^{n+\frac{1}{2}}\right)^{Q+2} \delta.
\end{equation}
We claim 
\begin{equation}\label{claimbigO}
\left|G^{(0)}_s\right| \gtrsim s^2 \quad \text{ as }  s \rightarrow 0^+
\end{equation}
To see this, consider the point $\bar{\zeta}=(\bar{\xi}, \bar{\tau})\in\de_p \Q^{(0)}$ such that $u(\bar{\zeta})=2=\sup_{\Q^{(0)}} u$. By continuity of $u$, there exists a small neighborhood $U_{\bar{\xi}}\times(\bar{\tau}-\theta^2,\bar{\tau}+\theta^2)$ of $\bar{\zeta}$ in which $u>1$. We face two possibilities: either $\bar{\zeta}$ is on the ``base'' of the cylinder $\Q^{(0)}$, in which case $\bar{\xi}\in \overline{B_{r\frac{\sigma_0}{2}}(0)}$ and $\bar{\tau}=-\frac{3}{4}b_Br^2$, or $\bar{\zeta}$ is on the ``lateral side'' of $\Q^{(0)}$, in which case $|\bar{\xi}|_B=r\frac{\sigma_0}{2}$ and $\bar{\tau}\in(-\frac{3}{4}b_Br^2,-\frac{1}{2}b_Br^2]$. In the first case, it suffices to notice that, up to restricting $U_{\bar{\xi}}$ and for $s$ small enough with respect to $\theta$ and $r$, we have $U_{\bar{\xi}}\times(\bar{\tau}-\frac{1}{4}b_Br^2s^2,\bar{\tau})\subset G^{(0)}_s$ and thus $\left|G^{(0)}_s\right| \gtrsim s^2$. In the second case, up to restricting $\theta$ and for $s$ small enough, we have instead a sector $C_s\subset B_{r\frac{\sigma_0}{2}(1+s)}(0)\smallsetminus B_{r\frac{\sigma_0}{2}}(0)$ such that $C_s\times (\bar{\tau}-\theta^2,\bar{\tau})\subset G^{(0)}_s$. Since $|B_{r\frac{\sigma_0}{2}(1+s)}(0)\smallsetminus B_{r\frac{\sigma_0}{2}}(0)|\sim s$ by the dilation properties, we deduce that $\left|G^{(0)}_s\right| \gtrsim s\geq s^2$. This completes the proof of the claim \eqref{claimbigO}.

By \eqref{valueatonehalf} and \eqref{claimbigO}, there exists $s_1 \in (0,1/2)$ such that $$\left|G^{(0)}_{s_1}\right| =  \left(C_2 \cdot r  \cdot s_1^{n+\frac{1}{2}}\right)^{Q+2} \delta.$$
Let $\zeta_0 \in \partial_p \Q^{(s_1/2)}$ be such that $u(\zeta_0) \geq 2$. Using Lemma \ref{incylinder}, item $(ii)$ with $\delta_1 = 0$, $\delta_2 = s_1$, we obtain the existence of a cylinder $\Q_{\rho}^{-b_B\rho^2,0}(\zeta_0) \subseteq \Q^{(s_1)} \backslash \Q^{(0)}$, where
$$\rho = C_2 \cdot  r \cdot s_1^{n+\frac{1}{2}}.$$
Define $\Q_{(0)} := \Q_{\rho}^{-b_B\rho^2,0}(\zeta_0)$ and $D_{(0)}:=G \cap \Q_{(0)}$. Notice that $\zeta_0 \in D_{(0)}$ and
$$|D_{(0)}| \leq \left|G^{(0)}_{s_1}\right| = \delta \rho^{Q+2}.$$
Consider the function $v := u - 1$. The measure estimate for $D_{(0)}$ above allows us to apply Lemma \ref{lemma91} to $v$. Noticing that $v(\zeta_0) \geq 1$, we thus conclude 
$$\sup_{D_{(0)}} u \geq \sup_{D_{(0)}} v \geq M.$$
This implies, by the weak maximum principle \eqref{MP}, that 
$$\sup_{\partial_p \Q^{(s_1)}} u \geq M.$$

{\bf Step 2:} The construction described above is the $(\ell=0)$-case of our finite iteration scheme: we started from the cylinder $\Q^{(0)}$ for which we have $\sup_{\partial_p \Q^{(0)}} u=2$, and we found $0<s_1<1$ (in fact $s_1\in (0,1/2)$) such that $\sup_{\partial_p \Q^{(s_1)}} u \geq M$. Suppose now that for some index $\ell \geq 1$, we have chosen $s_\ell\in(0,1)$ satisfying
$$\sup_{\partial_p \Q^{(s_{\ell})}} u \geq 2 \left(\frac{M}{2}\right)^{\ell}.$$
If $s_{\ell} \geq 1/2$, we proceed directly to step 3 (we know this cannot happen when $\ell=1$). Otherwise we have $1 > 1 - s_{\ell} > 1/2$. In this scenario, we show how to choose $s_{\ell + 1} > s_{\ell}$ such that 
$$\sup_{\partial_p \Q^{(s_{\ell+1})}} u \geq 2 \left(\frac{M}{2}\right)^{\ell+1}.$$
Let $G_{\ell} := \left\{z \in \Q^3_r : u(z) > \left(\frac{M}{2}\right)^{\ell}\right\}$. For any $s > 0$, define the family of sets
$$G^{(\ell)}_s:=G_{\ell} \cap\left(\Q^{(s + s_{\ell})} \backslash \Q^{(s_{\ell})}\right), \qquad 0 < s < 1 - s_{\ell}.$$
Since $G_{\ell} \subset G$, we have
$$\left|G^{(\ell)}_{1/2} \right| \leq |G| \leq  \left(C_2 r\left(\frac{1}{2} \right)^{n+\frac{1}{2}}\right)^{Q+2} \delta.$$
Arguing as in \eqref{claimbigO}, we also have $\left|G^{(\ell)}_s\right| \gtrsim s^2$ as $s \rightarrow 0^+$. Hence, there exists $\rho_{\ell } \in (0, \frac{1}{2}) $ such that 
$$\left|G^{(\ell)}_{\rho_{\ell}}\right| =  \left(C_2 \cdot r \cdot \rho_{\ell}^{n+\frac{1}{2}} \right)^{Q+2} \delta.$$
Let $\zeta_{\ell} \in \partial_p \Q^{(s_{\ell} + \frac{\rho_{\ell}}{2})}$ be such that $u(\zeta_{\ell}) \geq 2 \left(\frac{M}{2}\right)^{\ell}$. Defining $s_{\ell+1} := s_{\ell} + \rho_{\ell}$ and using Lemma \ref{incylinder}. item $(ii)$ with $\delta_1 = s_{\ell}$, $\delta_2 = s_{\ell+1}$, we obtain the existence of a cylinder $\Q_{\rho}^{-b_B\rho^2,0}(\zeta_{\ell}) \subseteq \Q^{(s_{\ell + 1})} \backslash \Q^{(s_{\ell})}$, where
$$\rho = C_2 \cdot  r \cdot \rho_{\ell}^{n+\frac{1}{2}}.$$
Define $\Q_{(\ell)} := \Q_{\rho}^{-b_B\rho^2,0}(\zeta_{\ell})$ and $D_{(\ell)}:= G \cap \Q_{(\ell)}$. Notice that $\zeta_{\ell} \in D_{(\ell)}$ and
$$|D_{(\ell)}| \leq \left|G^{(\ell)}_{\rho_{\ell}}\right| =  \left(C_2 \cdot r \cdot \rho_{\ell}^{n+\frac{1}{2}} \right)^{Q+2} \delta = \delta \rho^{Q+2}.$$
Consider the function $v := u - \left(\frac{M}{2}\right)^{\ell}$. The measure estimate for $D_{(\ell)}$ above allows us to apply Lemma \ref{lemma91} to $v$. Noticing that $v(\zeta_{\ell}) \geq \left(\frac{M}{2}\right)^{\ell}$, we thus conclude 
$$\sup_{D_{(\ell)}} u \geq \sup_{D_{(\ell)}} v \geq M \cdot \left(\frac{M}{2}\right)^{\ell} = 2 \left(\frac{M}{2}\right)^{\ell +1}.$$
This implies, by the weak maximum principle, that 
$$\sup_{\partial_p \Q^{(s_{\ell+1})}} u \geq 2 \left(\frac{M}{2}\right)^{\ell +1}.$$

{\bf Step 3:} There must exist a smallest integer $k \geq 1$ such $s_{k+1} \geq 1/2$, for otherwise the function $u$ would be unbounded on $\Q^3_r$. This implies Step 2 must terminate after finitely many iterations. By denoting $\rho_0:=s_1$ and recalling the definition of $s_{k+1}$, we have $\rho_0 + \rho_1 + \cdots + \rho_k \geq \frac{1}{2}$ and $\rho_0 + \cdots + \rho_{k-1} < \frac{1}{2}$. For each $\ell \in \left\{0, \ldots, k \right\}$, we know that the corresponding set $G^{(\ell)}_{\rho_{\ell}}$ satisfies
$$\left|G^{(\ell)}_{\rho_{\ell}}\right| =  \left(C_2 \cdot r \cdot \rho_{\ell}^{n+\frac{1}{2}} \right)^{Q+2} \delta,$$
and that $u > \left(\frac{M}{2}\right)^{\ell}$ on $G^{(\ell)}_{\rho_{\ell}}$ by definition. Since $\rho_0 + \cdots + \rho_k \geq \frac{1}{2}$, there must exist at least one index $i_0 \in\left\{0, \ldots, k \right\}$ such that 
$$\rho_{i_0} \geq \left(\frac{1}{2}\right)^{i_0 + 2}.$$
Therefore, we have
$$\left|G^{(i_0)}_{\rho_{i_0}}\right| \geq \left(C_2 \cdot r \cdot 2^{-(i_0 + 2)(n+\frac{1}{2})} \right)^{Q+2} \delta,$$
and
$$u > \left(\frac{M}{2}\right)^{i_0} \ \text{on} \ G^{(i_0)}_{\rho_{i_0}}.$$
We now make one final use of Theorem \ref{growthlemma}. Consider the function $v := \left(\frac{M}{2}\right)^{i_0} -u$. Then $\mathcal{L}_A v = 0$ and $v \leq \left(\frac{M}{2}\right)^{i_0}$ since $u$ is non-negative on $\Q^1_r$. Define $D := \left\{z \in \Q^1_r : v(z) > 0 \right\}$. Then $G_{i_0} \subset \Q^3_r \backslash D$. Since $G^{(i_0)}_{\rho_{i_0}} \subset G_{i_0}$, we have from the measure estimate above that
$$\dfrac{|\Q^3_r \backslash D|}{|\Q^3_r|} \geq \dfrac{\left(C_2 \cdot r \cdot 2^{-(i_0 + 2)(n+\frac{1}{2})} \right)^{Q+2} \delta}{|\Q^3_1| r^{Q+2}} = \frac{C_2^{Q+2}\delta}{|\Q^3_1|} 2^{-(n+\frac{1}{2})(Q+2)(i_0+2)}.$$
Finally, we may assume that $\left\{v \geq 0 \right\} \cap \Q^2_r \neq \emptyset$; for otherwise, we would have $u \geq \left(\frac{M}{2}\right)^{i_0} \geq 1$ on $\Q^2_r\supset \Q_r^+$, and \eqref{harnackH1} would automatically follow. Thus, we may apply Theorem \ref{growthlemma} to $v$ and obtain
$$\left(\frac{M}{2}\right)^{i_0} \geq \sup_{\Q^1_r} v \geq \left(1 + \eta \frac{C_2^{Q+2}\delta}{|\Q^3_1|} 2^{-(n+\frac{1}{2})(Q+2)(i_0+2)} \right) \sup_{\Q^2_r\cap D} v \geq  \left(1 + \eta \frac{C_2^{Q+2}\delta}{|\Q^3_1|} 2^{-(n+\frac{1}{2})(Q+2)(i_0+2)} \right) \sup_{\Q_r^+}v .$$
Inserting the definition of $v$ and recalling that $M = 2^{1+(n+\frac{1}{2})(Q+2)}$, we have
$$\left(\frac{M}{2}\right)^{i_0} \geq \left(1 + \eta \frac{C_2^{Q+2}\delta}{|\Q^3_1|} 2^{-2(n+\frac{1}{2})(Q+2)}\left(\frac{M}{2}\right)^{-i_0} \right) \left(\left(\frac{M}{2}\right)^{i_0} - \inf_{\Q_r^+} u \right).$$
Denoting $\hat{c} := \eta \frac{C_2^{Q+2}\delta}{|\Q^3_1|} 2^{-2(n+\frac{1}{2})(Q+2)}$, we get
$$\left(1+ \hat{c} \left(\frac{M}{2}\right)^{-i_0}\right) \inf_{\Q_r^+} u \geq \hat{c}.$$
Since $\left(\frac{M}{2}\right)^{-i_0} \leq 1$, we conclude that
$$\inf_{\Q_r^+} u \geq \frac{\hat{c}}{1 + \hat{c}}.$$
This establishes \eqref{harnackH1} when Case 2 holds, and finishes the proof of the theorem. \end{proof}

We next show the proof of Theorem \ref{harnackineqH2}.

\begin{proof}[Proof of Theorem \ref{harnackineqH2}] 
For small radii we can follow the proof of Theorem \ref{harnackineqH1}. In fact, if
$$ r\leq r_0:=\frac{\e_0}{K}\min\left\{1,\frac{1}{C_1C_2}4^{n+\frac{1}{2}}\right\}, $$
then we can invoke Theorem \ref{growthlemmaH2} and argue exactly as in the proofs of Lemma \ref{lemma91} and Theorem \ref{harnackineqH1}. This gives the existence of a structural constant $C_H$ such that, for all $\Q^1_r=\Q_{Kr}^{-b_B r^2,0}(z_0)\Subset\Omega$ and any $u \in C^2(\Omega)$ which is a nonnegative solution to $\elle_A u=0$ in $\Q^1_r$, we have
\begin{equation}\label{toiterate}
\sup_{\Q_{\frac{\sigma_0}{2}r}^{-\frac{3b_B}{4}r^2,-\frac{b_B}{2}r^2}(z_0)} u \leq C_H \inf_{\Q_{\frac{\sigma_0}{2}r}^{-\frac{b_B}{4}r^2,0}(z_0)} u,\qquad\mbox{ for any }0<r\leq r_0.
\end{equation}
For $r_0 < r \leq 1$, one can use \eqref{toiterate} along with the existence of Harnack chains established in this context by Polidoro \cite[Section 3]{PoliA}. This proves \eqref{harnackH2} with some constants $\sigma<\frac{\sigma_0}{2}$ and $C>C_H$ depending on the modulus of continuity of the coefficients $\omega$ (i.e. on $\e_0$).
\end{proof}

Finally, we provide the proof of Lemma \ref{incylinder}, as promised.

\begin{proof}[Proof of Lemma \ref{incylinder}] We first prove $(i)$. Fix $z_0=(x_0, t_0)$ in the closure of the cylinder $\Q_{R\left(\frac{1}{2}+\delta_1\right)}^{-b_B R^2\left(\frac{1}{2}+\delta_1\right),0}$, and fix any point $\bar{\zeta}=(\bar{\xi},\bar{\tau})\in \Q_{K\rho}^{-b_B\rho^2,0}(z_0)$. This means that there exists $\zeta=(\xi,\tau)\in \Q_{K\rho}^{-b_B\rho^2,0}$ such that $\bar{\zeta}=z_0\circ \zeta = (\xi + E(\tau)x_0, \tau + t_0)$. By definition we have
$$0>\tau + t_0>-b_B\rho^2 -b_B R^2\left(\frac{1}{2}+\delta_1\right) \geq -b_B R^2\left(\frac{1}{2}+\delta_2\right),$$
where the last inequality holds provided that
\begin{equation}\label{euno}
\rho\leq R\sqrt{\delta_2-\delta_1}.    
\end{equation}
We have also to find conditions ensuring that $|\xi + E(\tau)x_0|_B<R\left(\frac{1}{2}+\delta_2\right)$. This is trivial if $x_0=0$. So, suppose $x_0\neq 0$. We start by noticing that
$$\max{\left\{\left|D_{\frac{\bar{\sigma}}{|x_0|_B}} x_0\right|, \left|D_{\frac{\bar{\sigma}}{|x_0|_B}} x_0\right|^{\frac{1}{2n+1}}\right\}}\geq \frac{1}{\bar{\sigma}}\left|D_{\frac{\bar{\sigma}}{|x_0|_B}} x_0\right|_B=1,$$
which yields
\begin{equation}\label{eee}
\left|D_{\frac{\bar{\sigma}}{|x_0|_B}} x_0\right|^{\frac{1}{2n+1}}=\min{\left\{\left|D_{\frac{\bar{\sigma}}{|x_0|_B}} x_0\right|, \left|D_{\frac{\bar{\sigma}}{|x_0|_B}} x_0\right|^{\frac{1}{2n+1}}\right\}}\leq\frac{1}{\sigma_0}\left|D_{\frac{\bar{\sigma}}{|x_0|_B}} x_0\right|_B=\frac{\bar{\sigma}}{\sigma_0}.
\end{equation}
By \eqref{commu}, \eqref{benne}, and \eqref{eee}, we get
\begin{eqnarray*}
\left|\left(E(\tau)-\mathbb{I}_N\right) x_0\right|_B&=& \frac{|x_0|_B}{\bar{\sigma}} \left| D_{\frac{\bar{\sigma}}{|x_0|_B}} \left(E(\tau)-\mathbb{I}_N\right) x_0\right|_B = \frac{|x_0|_B}{\bar{\sigma}} \left|\left(E\left(\frac{\bar{\sigma}^2\tau}{|x_0|^2_B}\right)-\mathbb{I}_N\right) D_{\frac{\bar{\sigma}}{|x_0|_B}}x_0\right|_B\\
&\leq&c(n,B)\frac{|x_0|_B}{\bar{\sigma}} \max\left\{\left|D_{\frac{\bar{\sigma}}{|x_0|_B}}x_0\right|^{\frac{1}{3}},\left|D_{\frac{\bar{\sigma}}{|x_0|_B}}x_0\right|^{\frac{1}{2n+1}}\right\}\max\left\{\left|\frac{\bar{\sigma}^2\tau}{|x_0|^2_B}\right|^{\frac{1}{2n+1}},\left|\frac{\bar{\sigma}^2\tau}{|x_0|^2_B}\right|^{\frac{n}{2n+1}}\right\}\\
&\leq& R\frac{c(n,B)}{\bar{\sigma}}\left(\frac{\bar{\sigma}}{\sigma_0}\right)^{\frac{2n+1}{3}} \frac{|x_0|_B}{R} \max\left\{\left(\frac{R}{|x_0|_B}\right)^{\frac{2}{2n+1}}\left|\frac{\bar{\sigma}^2\tau}{R^2}\right|^{\frac{1}{2n+1}},\left(\frac{R}{|x_0|_B}\right)^{\frac{2n}{2n+1}}\left|\frac{\bar{\sigma}^2\tau}{R^2}\right|^{\frac{n}{2n+1}}\right\}\\
&\leq&R\frac{c(n,B)}{\bar{\sigma}}\left(\frac{\bar{\sigma}}{\sigma_0}\right)^{\frac{2n+1}{3}} \left(\frac{|x_0|_B}{R}\right)^{\frac{1}{2n+1}} \max\left\{\left|\frac{\bar{\sigma}^2\tau}{R^2}\right|^{\frac{1}{2n+1}},\left|\frac{\bar{\sigma}^2\tau}{R^2}\right|^{\frac{n}{2n+1}}\right\},
\end{eqnarray*}
where we have used $|x_0|_B\leq R $ (since $\delta_1\leq \frac{1}{2}$) and $n\geq 1$. If in addition
\begin{equation}\label{edue}
\rho\leq \frac{R}{\bar{\sigma}\sqrt{b_B}},
\end{equation}
then we also have $\left|\frac{\bar{\sigma}^2\tau}{R^2}\right|\leq \frac{\bar{\sigma}^2 b_B\rho^2}{R^2}\leq 1$, and so
\begin{equation}\label{valbe}
\left|\left(E(\tau)-\mathbb{I}_N\right) x_0\right|_B\leq R\frac{c(n,B)}{\bar{\sigma}}\left(\frac{\bar{\sigma}}{\sigma_0}\right)^{\frac{2n+1}{3}} \left(\bar{\sigma}^2 b_B\frac{\rho^2}{R^2}\right)^{\frac{1}{2n+1}}.
\end{equation}
Hence, by the definition of $\xi, x_0$, and by \eqref{triangleB}, \eqref{valbe}, we get
$$|\xi + E(\tau)x_0|_B\leq |\xi|_B + |x_0|_B + \left|\left(E(\tau)-\mathbb{I}_N\right) x_0\right|_B< K\rho + R\left(\frac{1}{2}+\delta_1\right) + R\frac{c(n,B)}{\bar{\sigma}}\left(\frac{\bar{\sigma}}{\sigma_0}\right)^{\frac{2n+1}{3}} \left(\bar{\sigma}^2 b_B\frac{\rho^2}{R^2}\right)^{\frac{1}{2n+1}}.$$
Finally, if we have
\begin{equation}\label{etre}
\rho\leq R \frac{\delta_2-\delta_1}{2K}\qquad\mbox{ and }
\end{equation}
\begin{equation}\label{equattro}
\rho\leq R\frac{1}{\bar{\sigma} \sqrt{b_B}}\left(\frac{\delta_2-\delta_1}{2}\frac{\sigma_0^{\frac{2n+1}{3}}}{c(n,B)\bar{\sigma}^{\frac{2n-2}{3}}}\right)^{n+\frac{1}{2}},
\end{equation}
then
$$|\xi + E(\tau)x_0|_B< R(\delta_2-\delta_1) + R\left(\frac{1}{2}+\delta_1\right)=R\left(\frac{1}{2}+\delta_2\right)$$
as desired. The conditions which $\rho$ has to satisfy are \eqref{euno}, \eqref{edue}, \eqref{etre}, and \eqref{equattro}. Since $\delta_2-\delta_1<1$, these conditions are satisfied if 
$$\rho\leq C_1 R (\delta_2-\delta_1)^{n+\frac{1}{2}}$$
for a suitable $C_1$ depending on $b_B, \sigma_0, \bar{\sigma}, K, c(n,B)$.\\
Let us now prove $(ii)$. Proceeding verbatim as in the first part, we can prove the existence of a structural constant $\tilde{C}_2$ such that, for any $\rho\leq \tilde{C}_2 R \left(\delta_2-\delta_1\right)^{n+\frac{1}{2}}$ and any $z_0$ in the closure of $\Q_{R\frac{\sigma_0}{2}\left(1+\frac{\delta_1+\delta_2}{2}\right)}^{-\frac{b_B}{4} R^2\left(3+\frac{(\delta_1+\delta_2)^2}{4}\right),-\frac{b_B}{2} R^2}$, we have
$$z_0\circ\zeta \in \Q_{R\frac{\sigma_0}{2}\left(1+\delta_2\right)}^{-\frac{b_B}{4} R^2\left(3+\delta^2_2\right),-\frac{b_B}{2} R^2}\quad\mbox{ for every fixed }\zeta\in \Q_{\rho}^{-b_B\rho^2,0}.$$
In particular, for such $\rho$ and using $\delta_1,\delta_2\leq 1$, we also have \eqref{valbe}. If we assume, in addition, that $z_0\in\de_p\Q_{R\frac{\sigma_0}{2}\left(1+\frac{\delta_1+\delta_2}{2}\right)}^{-\frac{b_B}{4} R^2\left(3+\frac{(\delta_1+\delta_2)^2}{4}\right),-\frac{b_B}{2} R^2}$, then either $t_0=-\frac{b_B}{4} R^2\left(3+\frac{(\delta_1+\delta_2)^2}{4}\right)$ or $|x_0|_B=R\frac{\sigma_0}{2}\left(1+\frac{\delta_1+\delta_2}{2}\right)$. If the first possibility occurs, we have 
$$\tau + t_0<t_0<-\frac{b_B}{4} R^2\left(3+\delta_1^2\right) \quad\mbox{ which implies }\quad z_0\circ\zeta \notin \Q_{R\frac{\sigma_0}{2}\left(1+\delta_1\right)}^{-\frac{b_B}{4} R^2\left(3+\delta^2_1\right),-\frac{b_B}{2} R^2}.$$
On the other hand, if the second possibility occurs, then by \eqref{triangleB} and \eqref{valbe} we get
\begin{eqnarray*}
|\xi + E(\tau)x_0|_B&\geq&  |x_0|_B - |\xi|_B - \left|\left(E(\tau)-\mathbb{I}_N\right) x_0\right|_B\\
&>& R\frac{\sigma_0}{2}\left(1+\frac{\delta_1+\delta_2}{2}\right) - \rho - R\frac{c(n,B)}{\bar{\sigma}}\left(\frac{\bar{\sigma}}{\sigma_0}\right)^{\frac{2n+1}{3}} \left(\bar{\sigma}^2 b_B\frac{\rho^2}{R^2}\right)^{\frac{1}{2n+1}}.
\end{eqnarray*}
If $\rho$ also satisfies
$$\rho\leq R\frac{\sigma_0}{2} \frac{\delta_2-\delta_1}{4}\quad\mbox{ and }\quad\rho\leq R\frac{1}{\bar{\sigma} \sqrt{b_B}}\left(\frac{\sigma_0}{2}\frac{\delta_2-\delta_1}{4}\frac{\sigma_0^{\frac{2n+1}{3}}}{c(n,B)\bar{\sigma}^{\frac{2n-2}{3}}}\right)^{n+\frac{1}{2}},$$
then
$$|\xi + E(\tau)x_0|_B>R\frac{\sigma_0}{2}\left(1+\delta_1\right) \quad\mbox{ which implies }\quad z_0\circ\zeta \notin \Q_{R\frac{\sigma_0}{2}\left(1+\delta_1\right)}^{-\frac{b_B}{4} R^2\left(3+\delta^2_1\right),-\frac{b_B}{2} R^2}.$$
Therefore, up to modifying the constant $\tilde{C}_2$ to a suitable structural constant $C_2$, we have the desired conclusion.
\end{proof}

\section*{Acknowledgments}

\noindent F.A. wishes to thank Prof. Brian Rider for providing financial support through his NSF grant DMS--1406107. G.T. has been partially supported by the Gruppo Nazionale per l'Analisi Matematica, la Probabilit\`a e le loro Applicazioni (GNAMPA) of the Istituto Nazionale di Alta Matematica (INdAM). The authors would like to thank Prof. Luis Silvestre for suggesting this problem at the 2017 Chicago Summer School in Analysis, and the anonymous referee for their valuable comments on the manuscript.

\bigskip

\noindent {\bf Compliance with Ethical Standards:} The authors declare that they have no conflict of interest.

\bibliographystyle{amsplain}

\end{document}